\documentclass[twoside,11pt]{article}

	\usepackage[T1]{fontenc}	
	\usepackage[latin1]{inputenc}	
	\usepackage{vmargin}		
	\usepackage{fancyhdr}		
	\usepackage{amsfonts,amsthm,amsmath,stmaryrd, amssymb,mathrsfs} 
	\usepackage{graphicx, subfigure} 
	\usepackage{mfpic} 
	\usepackage{listings} 
	\setpapersize{custom}{21cm}{29.7cm}
	\setmarginsrb{25mm}{10mm}{25mm}{47mm}{10mm}{6mm}{0mm}{10mm}
	\usepackage[pdftex]{hyperref}
		\usepackage{color}
	
	\numberwithin{equation}{section}
	
\def\e{\varepsilon}

\def\R{{\mathbb R}}
\def\C{{\mathbb C}}
\def\N{{\mathbb N}}
\def\Z{{\mathbb Z}}
\def\d{\partial}
\def\a{\alpha}
\def\EE{{\mathbb E}}
\def\g{\gamma}

\def\u{{\mathbf u}}
\def\v{{\mathbf v}}
\def\w{{\mathbf w}}
\def\f{{\mathbf f}}
\def\h{{\mathbf h}}
\def\H{{\mathbf H}}
\def\A{{\mathbf A}}
\def\F{{\mathbf F}}

\def\dsp{\displaystyle}
\def\be{\begin{equation}}
\def\ee{\end{equation}}
\def\Id{{\rm Id}}

\def\und{\underline}	

\begin{document}

\theoremstyle{plain}
\newtheorem{theo}{Theorem}

\theoremstyle{plain}
\newtheorem*{theo*}{Theorem}

\theoremstyle{plain}
\newtheorem{prop}{Proposition}[section]

\theoremstyle{plain} 
\newtheorem{defi}[prop]{Definition}

\theoremstyle{plain}
\newtheorem{coro}[prop]{Corollary}

\theoremstyle{plain}
\newtheorem{lemma}[prop]{Lemma}

\theoremstyle{definition}
\newtheorem*{condi}{Condition}

\theoremstyle{plain}
\newtheorem{notation}[prop]{Notation}

\theoremstyle{plain} 
\newtheorem{remark}[prop]{Remark}

\theoremstyle{plain} 
\newtheorem{hypo}[prop]{Assumption}

\title{On hyperbolicity and Gevrey well-posedness. \\ Part one: the elliptic case.}

\author{Baptiste Morisse \thanks{Université Paris Diderot, Sorbonne Paris Cité, Institut de Mathématiques de Jussieu-Paris Rive Gauche, UMR 7586, F- 75025, France - baptiste.morisse@imj-prg.fr.
The author thanks his advisor Benjamin Texier for all the remarks on this work.}
}
\date{\today}

\maketitle

\begin{abstract}
	In this paper we prove that the Cauchy problem for first-order quasi-linear systems of partial differential equations is ill-posed in Gevrey spaces, under the assumption of an initial ellipticity. The assumption bears on the principal symbol of the first-order operator. Ill-posedness means instability in the sense of Hadamard, specifically an instantaneous defect of Hölder continuity of the flow from $G^{\sigma}$ to $L^2$, where $\sigma\in(0,1)$ depends on the initial spectrum. Building on the analysis carried out by G. Métivier [\textit{Remarks on the well-posedness of the nonlinear Cauchy problem}, Contemp. Math. 2005], we show that ill-posedness follows from a long-time Cauchy-Kovalevskaya construction of a family of exact, highly oscillating, analytical solutions which are initially close to the null solution, and which grow exponentially fast in time. A specific difficulty resides in the observation time of instability. While in Sobolev spaces, this time is logarithmic in the frequency, in Gevrey spaces it is a power of the frequency. In particular, in Gevrey spaces the instability is recorded much later than in Sobolev spaces.
\end{abstract}


\setcounter{tocdepth}{2}
\tableofcontents




\section{Introduction}

We consider the Cauchy problem for first-order quasi-linear systems of partial differential equations
\be
	\label{Cauchy}
	\d_{t}u = \sum_{j=1}^{d} A_{j}(t,x,u) \d_{x_j}u + f(t,x,u) \;,\qquad u(0,x) = h(x)
\ee

\noindent where $t\geq 0$, $x\in\R^{d}$, $u(t,x)$ and $f(t,x,u)$ are in $\R^N$ and $A_j(t,x,u)\in\R^{N\times N}$. We assume throughout the paper that the $A_j$ and $f$ are analytic in a neighborhood of $(0,0,0)$.

We prove that if the first-order operator is initially micro-locally elliptic, then the Cauchy problem \eqref{Cauchy} is ill-posed in Gevrey spaces. Our results extend Métivier's ill-posedness theorem \cite{metivier2005remarks} for initially elliptic operators in Sobolev spaces.

While it may seem natural that Gevrey regularity, with associated sub-exponential Fourier rates of decay $O\left( e^{-|\xi|^{\sigma}} \right)$, with $\sigma <1$, will not be sufficient to counteract the exponential growth of elliptic operators (think of $e^{t\xi}$, as is the case for the Cauchy-Riemann operator $\d_{t} + i\d_{x}$), the proof of ill-posedness requires a careful analysis of linear growth rates and linear and nonlinear errors. This ill-posedness result is Theorem \ref{theorem.2}, stated in Section \ref{subsection.result}. By ill-posedness, we mean the absence of a Hölder continuous dependence on the data, as measured from $G^{\sigma}$ to $L^2$. The precise definition is given in Section \ref{subsection.definitions}. The larger $\sigma$, the stronger the result. Of course, well-posedness holds in the limiting case $\sigma=1$, corresponding to analytic functions. Assuming only a property of micro-local ellipticity for the principal symbol of \eqref{Cauchy}, we obtain, in Theorem \ref{theorem.2}, the bound $\sigma < 1/(m + 1)$, where $m \geq 1$ is an algebraic multiplicity. Under an assumption of smooth partial diagonalization (see Assumption \ref{hypo.2}), we obtain, in Theorem \ref{theorem.3}, ill-posedness for any $\sigma <1/2$ regardless of the algebraic multiplicity. Under stronger spectral assumptions (see Assumption \ref{hypo.3}), we obtain, in Theorem \ref{theorem.4}, ill-posedness for any $\sigma < 2/3$ and we outline the conditions which allow for an instability proof at an arbitrarily high Gevrey regularity.

We note that an equation may be simultaneously ill-posed in Sobolev spaces and well-posed in Gevrey spaces (for instance, the Prandtl equation  \cite{gerard2010ill}, \cite{gerard2013well}). Besides well-posedness, the distinct but related phenomenon of Landau damping for Vlasov-Poisson occurs in Gevrey spaces \cite{mouhot2010landau}, \cite{bedrossian2013landau}, but not in Sobolev spaces \cite{bedrossian2016nonlinear}.

In the companion paper \cite{morisse2016II}, we extend these results to systems transitioning from hyperbolicity to ellipticity, following \cite{lerner2010instability} and \cite{lerner2015onset}.


\subsection{Background: on Lax-Mizohata results}


The question of the well-posedness of the Cauchy problem was first introduced and studied by Hadamard in \cite{hadamard1902problemes}. Hadamard proved, in the case of linear second-order elliptic equations, that the associated solution flow is not regular in the vicinity of any solution of the system. The case of linear evolution systems of the form \eqref{Cauchy}, with $A_j(t,x,u) \equiv A_j(t,x) $, $f(t,x,u) \equiv f(t,x)$ was first studied by Lax in \cite{lax2005asymptotic}, where the proof was given that hyperbolicity of the system, i.e. reality of the spectrum of the principal symbol, was a necessary condition for \eqref{Cauchy} to be well-posed in the sense of Hadamard in $C^k$ spaces. Lax's proof relied on separation of the spectrum. Mizohata extended Lax's result without this assumption in \cite{mizohata1961some}. Some cases of nonlinear systems were studied later by Wakabayashi  in \cite{wakabayashi2001lax} (here with stability also with respect to source term) and by Yagdjian in \cite{yagdjian1998note} and \cite{yagdjian2002lax} (there in the special case of gauge invariant systems).

A first statement of a precise Lax-Mizohata result for first-order quasi-linear systems was given by Métivier in \cite{metivier2005remarks}, with a precise description of the lack of regularity of the flow. As we will adapt the methods used by Métivier, we want to take a close look at \cite{metivier2005remarks}.

\subsection{On Métivier's result in Sobolev spaces}
\label{subsubsection.metivier}

In Section 3 of \cite{metivier2005remarks} Guy Métivier proves Hölder ill-posedness in Sobolev spaces for the Cauchy problem \eqref{Cauchy}, as soon as hyperbolicity fails at $t=0$. The initial defect of hyperbolicity means here that there are some $x_0\in\R^d$, $\vec{u}_0\in\R^N$ and $\xi_0\in\R^{d}$ such that the principal symbol evaluated at $(0,x_0,\vec{u}_0,\xi_0)$:
\be
	\label{intro.def.A0}
	A_0 := \sum_j A_j(0,x_0,\vec{u}_0)\xi_{0,j}
\ee

\noindent is supposed to have a couple of eigenvalues with non zero imaginary part, say $\pm i \g_0$, with eigenvectors $\vec{e}_{\pm}$. Hölder well-posedness, locally in time and space, would mean that initial data $h_1$ and $h_2$ in $H^{\sigma}(B_{r_0}(x_0))$, for some small $r_0>0$, would generate solutions $u_1$ and $u_2$ such that
\be
	\label{well}
	||u_2 - u_1||_{L^2(\Omega)} \lesssim || h_2 - h_1 ||_{H^{\sigma}(B_{r_0}(x_0))}^{\a}
\ee

\noindent for some space-time domain $\Omega$, for some $\sigma\geq 0$, some $\a\in(0,1]$. In order to disprove \eqref{well}, Métivier chooses $h_1 \equiv \vec{u}_0$, and lets $u_1$ the Cauchy-Kovalevskaya solution issued from $h_1$, the existence of which is granted, locally in space and time, by the analyticity assumption on the coefficients $A_j$ and $f$. Translating, Métivier is reduced to the case $\vec{u}_0=0$, $u_1 \equiv 0$, and the proof that \eqref{well} does not hold is reduced to the construction of a family $(u_{\e})_{\e >0}$ of initially small, exact analytical solutions such that 
\be
	\label{intro.holder}
	\lim_{\e\to0} \frac{||u_{\e}||_{L^2(\Omega_{\e})}}{||u_{\e}(0)||^{\a}_{H^{\sigma}(B_0(x_0))}} = +\infty 
\ee

\noindent for all Hölder exponent $\a \in (0,1]$ and all Sobolev indices $\sigma >0$, where $\Omega_{\e}$ is a small conical space-time domain centered at $(0,x_0)$. 

To highlight the specific frequency $\xi_0$ at which the initial ellipticity occurs, Métivier looks for solutions of the form
\be
	\label{intro.ansatz}
	u_{\e}(t,x) =  \e \u(t/\e, x, (x-x_0)\cdot\xi_0/\e) 
\ee

\noindent with $\e$ a small parameter and $\u(s,x,\theta)$ is periodic in $\theta$. Then $\u$ solves 
\be
	\label{intro.equation}
	\d_{s}\u - A_0\d_{\theta}\u = G(\e \u) 
\ee

\noindent where $A_0$ is defined by \eqref{intro.def.A0} and $G(\e \u)$ comprises both linear and nonlinear "errors" terms. Factorizing the propagator, an equivalent fixed point equation is obtained 
\be
	\label{intro.pointfixe}
	\u = e^{sA_0\d_{\theta}}\u(0) + \int_{0}^{s} e^{(s-s')A_0\d_{\theta}}G(\e \u(s'))ds' .
\ee

For equation \eqref{intro.pointfixe}, the goal is to prove:
\begin{itemize}
	\item The existence of solutions over the space-time domain $\Omega_{\e}$. This is a Cauchy-Kovalevskaya type of result, discussed in Section \ref{intro.ck}.
	\item The \textit{wild growth estimate} \eqref{intro.holder}. Since the instability develops in time, the existence domain $\Omega_{\e}$ must be large enough for \eqref{intro.holder} to be recorded. This point is discussed in Section \ref{intro.growth.solutions}.
\end{itemize}


\subsubsection{Exponential growth of the solutions}
\label{intro.growth.solutions}

 As a consequence of the assumption of ellipticity on $A_0$ defined by \eqref{intro.def.A0}, the propagator has an exponential growth in Fourier
\be
	\label{intro.growth.fourier}
	\left|\left(e^{(s-s')A_0\d_{\theta}} \u(s,x,\theta)\right)_n\right| \lesssim e^{|n|\g_0(s-s')} | \u_n(s,x)|
\ee

\noindent where we denote by $(\cdot)_n$ the $n$-th Fourier mode with respect to the periodic variable $\theta$. We recall that $\xi_0$ is the distinguished frequency for which $A_0$, defined in \eqref{intro.def.A0}, has a couple of non real eigenvalues associated with eigenvectors $\vec{e}_{\pm}$. We define well-chosen initial data 
\be
	\label{intro.initial}
	h_{\e} = \e^{M+1} \left( e^{\mp i x\cdot\xi_0/\e} \vec{e}_{\pm} \right) \quad , \quad \h_{\e} := \e^{M}\left(e^{ \mp i \theta}\vec{e}_{\pm}\right) 
\ee

\noindent for which the upper bound is attained:
\be
	\label{intro.fepsilon}
	\f_{\e}(s,\theta) := e^{sA_0\d_{\theta}}\h_{\e}(\theta) \quad \text{satisfies} \quad \left|\left(\f_{\e}\right)_n\right| \approx \e^{M}e^{\g_0s} \; , \quad \forall n\in\Z .
\ee

\noindent Above $ \f_{\e}(s,\theta) $ is the free solution of \eqref{intro.equation}, that is the solution of the equation when $G(\e \u) = 0$. One key observation in view of the Hadamard instability is that, for times of order $M|\ln(\e)|$, the free solution $ \f_{\e} $ is of order $1$ with respect to $\e$, whereas at time $0$ it is of order $\e^{M}$. Roughly there are $f_{\e}(t,x) = \f_{\e}(t/\e,x,(x-x_0)\cdot\xi_0/\e)$, $h_{\e}(t,x) = \h_{\e}(t/\e,x,(x-x_0)\cdot\xi_0/\e)$ and $\Omega_{\e}$ a small conical space-time domain that contains the ball $B_{\e}((M|\ln(\e)|,x_0))$ of $\R_s\times\R_x^{d}$ for which there holds
\be
	\label{intro.ratio}
	\frac{||f_{\e}||_{L^2(\Omega_{\e})}}{||h_{\e}||^{\a}_{H^{\sigma}}} \approx \e^{(d+1)/2} \e^{-\a(M-\sigma)} 
\ee

\noindent and a suitable choice of $M$ leads to \eqref{intro.holder} in the simplified case $u_{\e} = f_{\e}$, as $\e\to0$. 

Through a careful analysis of the quasilinear system, Métivier proved that the nonlinear solution $\u_{\e}$ is close enough to $\f_{\e}$ in such a way that the growth \eqref{intro.fepsilon} of the free solution $\f_{\e}$ in long time $O(|\ln(\e)|)$ passes on to solutions $\u_{\e}$, such that
\be
	\label{intro.growth}
	|\u_{\e}(s,x,\theta)| \gtrsim \e^{M} e^{\g_0 s} 
\ee
	
\noindent in a whole neighborhood of $(s,x) = (M|\ln(\e)|,x_0)$. This estimate from below leads finally to \eqref{intro.holder}.

In this sketch of analysis, we see in particular that the (projection over the temporal coordinate of) the existence domain $\Omega_{\e}$ introduced in Section \ref{subsubsection.metivier} must be large enough to contain time intervals $[0,M|\ln(\e)|]$. In Gevrey spaces, this domain must be much larger, see Section \ref{intro.extension}.

\subsubsection{Existence of solutions via a long-time Cauchy-Kovalevskaya result}
\label{intro.ck}

In order to show that nonlinear solution $\u_{\e}$ of equation \eqref{intro.pointfixe} actually exists for sufficiently long time $O(M|\ln(\e)|)$, Métivier proved a long-time Cauchy-Kovalevskaya theorem using techniques of majoring series developed by Wagschal in \cite{wagschal1979probleme} for the resolution of the nonlinear Goursat problem. A presentation of the method can also be found in \cite{cartan1995theorie}, and is developed extensively in Section \ref{subsection.majoring}.

For formal series $\phi(x) = \sum_{k\in\N^{d}} \phi_{k} x^{k}$ and $\psi(x) = \sum_{k\in\N^{d}} \psi_{k} x^{k}$, with $\psi_{k} \geq 0$, we define the relation
$$
	\phi \prec \psi \quad \Longleftrightarrow \quad |\phi_{k}| \leq \psi_k \; , \; \forall k \in \N^{d} .
$$

\noindent The method is based on the observation that, if $\psi$ has convergence radius $R^{-1} >0$ and $\phi \prec \psi$, then $\phi$ has a convergence radius at least equal to $R^{-1}$. Conversely, there are series of one variable $\Phi(z)$ with convergence radius equal to $1$ that satisfy the following property: for any series $\phi$ with convergence radius less than $R^{-1}$, there is $C>0$ such that $\phi \prec C \Phi(R \sum_j x_j)$. The norm of $\phi$ will be defined as the best constant $C$ (see Definition \ref{definition.EE_s}). An example is $\Phi(z) = \frac{1}{1 - z}$, which satisfies the previous property thanks to Cauchy's inequalities. 

Based on those two observations, the method consists in shifting the focus from $\phi$ to $\Phi$. The key is that $\Phi$ can be taken to be much simpler than the original, typically unknown, series. In this paper we choose $\Phi$ with convergence radius equal to one and also such that $\Phi^2 \prec \Phi$ (see point $4$ in Lemma \ref{majoring.properties} in Section \ref{subsection.majoring}).

Now assume that we are given an initial datum $u(0,\cdot)$ in \eqref{intro.pointfixe} such that $u(0,x) \prec \Phi(R \sum_j x_j)$. The Cauchy-Riemann operator $\d_t + i\d_x$ provides the simplest example of an elliptic Cauchy problem. On this example the radius of analyticity decays linearly in time: the datum $u$ with $\hat{u}(0,\xi) = e^{-R^{-1}|\xi|}$ generates the solution $\hat{u}(t,\xi) = e^{-(R^{-1}-t)\xi}$ , for $t > 0$ and $\xi > 0$. It makes sense to assume similarly a linearly decaying radius of convergence for the solutions to our elliptic problems. Thus after comparing $u(0)$ to $\Phi(R \sum_j x_j)$, we will compare $u(s)$ to $\Phi(R \sum_j x_j + \e\rho s)$, where $R$ and $\rho$ are parameters to be specified later. Note that the series $\Phi(R \sum_j x_j + \e\rho s)$ has converging radius $R^{-1}(1 - \e \rho s)$, which is non zero for $s < (\e \rho)^{-1}$ ; this is hence the \textit{maximal time of regularity} for the solutions.

%
%


For simplicity of exposition, consider equation \eqref{intro.pointfixe} with $G(\e\u) \equiv \e \sum_j A_j(\e s,x,\vec{u}_0)\d_{x_j}\u$ and $A_0 \equiv 0$. The right-hand side of \eqref{intro.pointfixe} reduces then to
\be
	\label{intro.def.rhs}
	\int_{0}^{s} \e \sum_j A_j(\e s',x, \vec{u}_0)\d_{x_j} \u(s') \,ds'.
\ee

\noindent By assumption of analyticity of the $A_j$, we may control the series $A_j(\e s',x, \vec{u}_0)$ by the model $\Phi(R \sum_j x_j + \e\rho s)$, up to a multiplicative constant. Then \eqref{intro.def.rhs} is controlled, in the sense of the binary relation $\prec$ and up to a multiplicative constant, by
\begin{align*}
	& \int_{0}^{s} \e \Phi(R \sum_j x_j + \e\rho s') \sum_j \d_{x_j} \Phi\left(R \sum_j x_j + \e\rho s'\right) \, ds' \\
	& \prec \int_{0}^{s} \e R \Phi(R \sum_j x_j + \e\rho s') \Phi'\left(R \sum_j x_j + \e\rho s'\right) \, ds' \\
	& \prec \int_{0}^{s} \e R \Phi'\left(R \sum_j x_j + \e\rho s'\right) \, ds' \\
	& \prec R\rho^{-1} \Phi\left(R \sum_j x_j + \e\rho s\right) .
\end{align*}

\noindent Above, we used $2\Phi \Phi' \prec \Phi'$, a consequence of $\Phi^2 \prec \Phi$ (the relation $\prec$ is compatible with derivation, see Lemma \ref{majoring.properties}). We observed above the phenomenon of regularization (of $\d_{x_j}$) by integration in time, as in \cite{ukai2001boltzmann}. The "error" \eqref{intro.def.rhs} is controlled at a cost of $R\rho^{-1}$.

To conclude to the existence of the family of analytic solutions $\u_{\e}$ exhibiting the growth \eqref{intro.growth} on sufficiently long time $O(M|\ln(\e)|)$, Métivier compared the maximal time of regularity $(\e\rho)^{-1}$, which then has to be greater than the instability time $M|\ln(\e)|$. This implies some constraints on $R$ and $\rho$, and finally on the domain of existence $\Omega_{\e}$. We will not go into more detail at this point, as those constraints will appear in the Gevrey analysis too.



\subsection{Extension to Gevrey spaces}
\label{intro.extension}

The aim of this article is to prove the same kind of Hölder ill-posedness as in \cite{metivier2005remarks}, under the assumption of analyticity of the coefficients of the $A_j$. But whereas \cite{metivier2005remarks} holds in Sobolev spaces, we prove here instability in Gevrey spaces\footnote{This has been suggested by Jeffrey Rauch, whom the author thanks warmly.}. Following Métivier's method, we construct a family of solutions $(u_{\e})_{\e}$ that satisfies
\be
	\label{intro.holder.gevrey}
	\lim_{\e\to0} \frac{||u_{\e}||_{L^2(\Omega)}}{||u_{\e}(0)||^{\a}_{G^{\sigma}(B_0)}} = +\infty 
\ee

\noindent where the Gevrey space $G^{\sigma}(B_0)$ is precisely defined in Section \ref{subsection.definitions}, with $B_0$ a ball of $\R^{d}$ containing the distinguished point $x_0$. Our goal in this Section is to informally describe the specific difficulties posed by the analysis in Gevrey spaces.

\subsubsection{On the time of instability in Gevrey spaces}
\label{subsubsection.time}

We first need to find a suitable replacement for the small coefficient $\e^{M}$ of $\h_{\e}$ defined in \eqref{intro.initial} in the Sobolev framework. Indeed, the highly oscillating function $e^{ix\cdot\xi_0/\e}$ has Sobolev norm $ ||e^{ix\cdot\xi_0/\e}||_{H^{\sigma}(B_0)} \approx \e^{-\sigma} $ whereas the Gevrey norm satisfies (see Definition \ref{def.gevrey} and Lemma \ref{size.gevrey.exp}) $ 	||e^{ix\cdot\xi_0/\e}||_{G^{\sigma}(B_0)} \approx e^{\e^{-\sigma}}$. Appropriate initial data are both small and highly oscillating. Thus we replace \eqref{intro.initial} by
\be
	\label{intro.initial.gevrey}
	h_{\e} = e^{-\e^{-\delta}} \left(e^{\mp i x\cdot\xi_0/\e} \vec{e}_{\pm} \right) \quad , \quad \h_{\e} := e^{-\e^{-\delta}}\left(e^{i\mp \theta}\vec{e}_{\pm}\right) 
\ee 

\noindent with $\sigma < \delta$. At the end of the analysis, we expect \eqref{intro.growth} to be replaced by
\be
	\label{intro.growth.gevrey}
	|\u_{\e}(s,x,\theta)| \gtrsim e^{-\e^{-\delta}} e^{\g_0 s} .
\ee

\noindent This leads to a typical observation time $ \e^{-\delta} $. This is the time for which the time exponential growth associated with the ellipticity counterbalances the very small initial amplitude. This observation time is far bigger than the typical Sobolev time $O(|\ln(\e)|)$ described above in Section \ref{subsubsection.metivier}. Note that the limitation $\sigma < \delta$ ensures at least formally that the ratio \eqref{intro.ratio} in Gevrey spaces $G^{\sigma}$ diverges as $\e\to0$ (see Remark \ref{remark.size.gevrey}).

\subsubsection{On the control of linear errors over long times}
\label{subsubsection.control}

Typically the estimates for $G(\e \u)$ (with notation introduced in \eqref{intro.equation}), which comprises both linear and nonlinear error terms, degrade over time. This is problematic in view of the resolution of the fixed point equation \eqref{intro.pointfixe}. By definition of $A_0$ in \eqref{intro.def.A0}, the linear error comprises term
$$
	(\sum_j A_j(\e s,x,\e\u) \xi_{0,j} - A_0 )\d_{\theta} \u \approx (\e s + |x-x_0| + \e \u)\d_{\theta}\u .
$$

Suppose now, for simplicity of exposition, that $G(\e \u) = \e s \d_{\theta} \u$, and recall that $s = O(\e^{-\delta})$ according to the sketch of analysis of Section \ref{subsubsection.time}. Suppose in addition that the linear bound \eqref{intro.growth.fourier} holds, and that we have an a priori control of the Fourier mode $n={}-1$ of the solution $\u$ with a growth rate that is equal to the linear growth rate
$$
	|\u_{{}-1}(s)| \lesssim e^{-\e^{-\delta}} e^{\g_0 s}.
$$

\noindent The amplitude $e^{-\e^{-\delta}}$ is the one previously discussed in Section \ref{subsubsection.time}. Then equation \eqref{intro.pointfixe} for the Fourier mode $n={}-1$ reduces to
$$
	\u_{{}-1}(s) - e^{-\e^{-\delta}}e^{\g_0 s}\vec{e}_{+} =  \int_{0}^{s} e^{{}-i(s-s')A_0} \left(\e s' ({}-i) \u_{-1}(s')\right) \,{\rm d}s'
$$

\noindent where $\vec{e}_{+}$ is the eigenvector of $A_0$ associated to the eigenvalue with imaginary part $i\g_0$. For the right-hand side, we have the estimate: 
\begin{eqnarray}
	\left| \int_{0}^{s} e^{{}-i(s-s')A_0} \left(\e s' (-i) \u_1(s')\right) \,{\rm d}s' \right| & \lesssim & \int_{0}^{s} e^{\g_0(s-s')} \left(\e s' e^{-\e^{-\delta}} e^{\g_0 s'}\right) \,{\rm d}s' \nonumber \\
		& \lesssim & \frac{1}{2} \e s^2 e^{-\e^{-\delta}}e^{\g_0 s} \label{intro.computation}
\end{eqnarray}

\noindent thanks to the upper bound \eqref{intro.growth.fourier}. Hence $\u_{{}-1}(s)$ would satisfy \eqref{intro.growth.gevrey} if $ \e s^2  = o_{\e\to0}(1) $ for any $s\in[0,\e^{-\delta})$, which would lead to the stringent constraint on the Gevrey index $ \sigma < \delta < 1/2 $.

Thus we need to consider the varying-coefficient operator $\sum_j A_j(\e s , x, \vec{u}_0)\xi_{0,j} \d_{\theta}$, as opposed to \cite{metivier2005remarks} where the constant-coefficient operator $A_0 \d_{\theta}$ was considered.

\subsubsection{On linear growth bounds}
\label{intro.section.bounds}

As discussed in Section \ref{subsubsection.control}, we need to work with the varying-coefficient operator 
$$
	\sum_j A_j(\e s , x, \vec{u}_0)\xi_{0,j} \d_{\theta}.
$$

\noindent We introduce first the propagator $U(s',s,x,\theta)$ which solves
$$
 	\d_{s} U(s',s,x,\theta) - \sum_j A_j(\e s , x, \vec{u}_0)\xi_{0,j} \,\d_{\theta} \, U(s',s,x,\theta) \quad , \quad U(s',s',x,\theta) \equiv {\rm Id}.
$$

\noindent As $\sum_j A_j(\e s , x, \vec{u}_0)\xi_{0,j}$ does not depend on $\theta$, the Fourier coefficients $U_n(s',s,x)$ of the propagator satisfies the ODE
$$
	\d_{s} U_n(s',s,x) - in\sum_j A_j(\e s , x, \vec{u}_0)\xi_{0,j} \,U_n(s',s,x) \quad , \quad U_n(s',s',x) \equiv {\rm Id}.
$$

\noindent Then $U(\theta)$ acts diagonally on each Fourier component. Note that in the autonomous case $\sum_j A_j(\e s , x, \vec{u}_0)\xi_{0,j} \equiv \sum_j A_j(0 , x, \vec{u}_0)\xi_{0,j}$, the propagator satisfies
$$
	U(s',s,x,\theta) = \exp \left( (s-s') \sum_j A_j(0 , x, \vec{u}_0)\xi_{0,j}  \, \d_{\theta} \right).
$$

\noindent Using the propagator $U(s,s,x,\theta)$, fixed point equation \eqref{intro.pointfixe} is replaced by
\be
	\label{intro.pointfixe.full}
	\u(s,x,\theta) = \f(s,x,\theta) + \int_{0}^{s} U(s',s,x,\theta) G(\e\u(s',x,\theta)) ds' 
\ee

\noindent where $\f(s,x,\theta) = U(0,s,x,\theta) \h_{\e}(\theta)$ is the free solution, with $\h_{\e}$ defined in \eqref{intro.initial.gevrey}.

For the $n$-th Fourier coefficient $U_n(s',s,x)$ of the propagator, the derivation of bounds is described for instance in Section 4 of \cite{lerner2015onset}. Eigenvalues may cross at the distinguished point $(0,x_0)$. In particular, eigenvalues and eigenprojectors may not be smooth, although eigenvalues are continuous. Since we do not want to formulate any additional assumption on the symbol besides ellipticity (although see Section \ref{intro.subsection.all.gevrey} below and Theorem \ref{theorem.3}), this forces us, in the derivation of upper bounds of $U_n(s',s,x)$, to resort to the procedure of approximate trigonalization described for instance in \cite{lerner2015onset}. 

In this procedure, a small error is produced in the rate of growth. On one side, an upper bound
\be
	\label{intro.propa.bis}
	\left|U_n(s',s,x)\right| \lesssim \omega^{-(m-1)} e^{|n|(s-s') ({\rm Im}\,\lambda_0 + R^{-1} + \e \und{s} + \omega)}
\ee

\noindent is achieved, where $\lambda_0$ is an eigenvalue of $A_0$ with positive imaginary part which is maximal among the other eigenvalues, and $m$ is the algebraic multiplicity of $\lambda_0$ in the spectrum. In \eqref{intro.propa.bis} the parameter $\omega>0$ is associated with the trigonalization error. The optimal choice of $\omega$ is described below in Section \ref{intro.endgame}. The bound \eqref{intro.propa.bis} holds for $x$ in $B_{R^{-1}}(x_0)$ and $s$ in $(0,\und{s})$, where $R^{-1}$ is the convergence radius and $\und{s}$ the final time of observation. This is made precise in Lemma \ref{lemma.growth.propa}.

On the other side, the free solution satisfies a bound of the form
\be
	\label{intro.lower}
	|\f_{\e}(s,x,\theta)| \gtrsim  \omega^{-(m-1)}\,e^{-\e^{-\delta}} e^{s({\rm Im}\,\lambda_0 - r - \e \und{s} - \omega)}
\ee

\noindent for $(s,x)\in(0,\und{s})\times B_{r}(x_0)$ with $r$ small. This is made precise in Lemma \ref{lemma.growth.free.solution}.

\subsubsection{On the endgame}
\label{intro.endgame}

As we did in Section \ref{subsubsection.control}, suppose now that there holds $G(\e\u) = \e \sum_j A_j(\e s , x \vec{u}_0) \d_{x_j}\u(s)$ and the linear bound \eqref{intro.propa.bis}. Suppose also that we have an a priori control of the Fourier mode $n=1$ of the solution $\u$ with a growth rate that is equal to the linear growth rate
\be
	\label{intro.apriori}
	|\u_1(s)| \lesssim e^{-\e^{-\delta}} \omega^{-(m-1)} e^{(s-s') ({\rm Im}\,\lambda_0 + R^{-1} + \e \und{s} + \omega)} .
\ee

\noindent In view of bound \eqref{intro.propa.bis} and equation \eqref{intro.pointfixe.full}, there holds then for the Fourier mode $n=1$ the bound
$$
	|\u_1(s) - \f_1(s)| \lesssim \int_{0}^{s} \omega^{-(m-1)} e^{(s-s') ({\rm Im}\,\lambda_0 + R^{-1} + \e \und{s} + \omega)} \e \sum_j A_j(\e s' , x \vec{u}_0) \d_{x_j}\u(s') ds' .
$$

\noindent Thanks to the majoring series method explained in Section \ref{intro.ck} and based on \eqref{intro.apriori} , we may expect to bound the above by
\be
	\label{intro.local}
	|\u_1(s) - \f_1(s)| \lesssim e^{-\e^{-\delta}} \omega^{-2(m-1)} e^{s ({\rm Im}\,\lambda_0 + R^{-1} + \e \und{s} + \omega)} R\rho^{-1} .
\ee

\noindent To end the proof, it would suffice then to show that $\u_1$ has the same bound from below as $\f_1$ in \eqref{intro.lower}. This is the case if the right-hand side of \eqref{intro.local} satisfies
\be
	\label{intro.intermediaire}
	e^{-\e^{-\delta}} \omega^{-2(m-1)} e^{s ({\rm Im}\,\lambda_0 + R^{-1} + \e \und{s} + \omega)} R\rho^{-1} \ll \omega^{-(m-1)}\,e^{-\e^{-\delta}}  e^{s({\rm Im}\,\lambda_0 - r - \e \und{s} - \omega)}
\ee

\noindent for all $s\in(0,\und{s})$, where $\ll$ is defined in \eqref{notation.ll}. This is equivalent to
\be
	\label{intro.constraint}
	\omega^{-(m-1)}\,e^{\und{s} (R^{-1} + r + \e \und{s} + \omega)} R\rho^{-1} \ll 1 .
\ee

\noindent As explained in Section \ref{subsubsection.time}, the final time $\und{s}$ is of order $\e^{-\delta}$. In order for \eqref{intro.constraint} to be satisfied, the argument of the exponential should be at most of order $1$ as $\e$ goes to $0$. Hence $R^{-1}$, $r$ and $\omega$ are chosen to be less than $\e^{\delta}$. Note that we also get once again the constraint $\e\und{s}^2 <1$, which brings back the limitation $\sigma < \delta < 1/2$ on the Gevrey index. 

Besides \eqref{intro.constraint}, another constraint shows up in the analysis. Recall that we work with the majoring series model $\Phi(R\sum_j x_j + \e\rho s)$. Its domain of analyticity is the conical space-time domain $\{ (s,x) \,|\, R\sum_j |x_j| + \e\rho s < 1 \}$. As the time of instability $\und{s}$ is of order $\e^{-\delta}$, in order to see the instability the maximal regularity time $(\e\rho)^{-1}$ has to be greater than $\e^{-\delta}$. Hence another constraint
\be
	\label{intro.constraint.2}
	\e^{1 - \delta} \ll \rho^{-1} .
\ee

\noindent Since $\omega$ and $R^{-1}$ are of order $\e^{\delta}$, we rewrite constraint \eqref{intro.constraint} as $\rho^{-1} \ll \e^{(m-1)\delta} R^{-1}$ and then as
$$
	\rho^{-1} \ll \e^{m\delta} .
$$

\noindent Finally we end up with a consistency inequality $\e^{1-\delta} \ll \e^{m\delta}$, equivalent to the limitation $\delta < 1/(m+1)$ of the Gevrey index. This is our principal result, detailed in Theorem \ref{theorem.2}. 

\subsubsection{On proving instability for higher Gevrey indices}
\label{intro.subsection.all.gevrey}

We saw above in Section \ref{intro.endgame} that, in the general case, the consideration of the varying-coefficient operator $ \sum_j A_j(\e \tau,x,\vec{u}_0) \xi_{0,j} \d_{\theta} $ does not free us from the constraint $ \sigma < 1/2$. Indeed, as discussed in Sec \ref{intro.endgame}, we actually need to impose $ \sigma < 1/(m+1)$, where $ m \geq 1$ is the algebraic multiplicity of $\lambda_0$ in the spectrum.

We describe here a situation in which we improve the limiting Gevrey index.

Assume finally that \eqref{intro.propa.bis} and \eqref{intro.lower} can be replaced by 
\be
	\label{intro.propa.bis.max}
	\left|U_n(s',s)\right| \lesssim  e^{|n|(s-s') ({\rm Im}\,\lambda_0 + \omega)}
\ee

\noindent and
\be
	\label{intro.lower.max}
	|\f_{\e}(s,x,\theta)| \gtrsim \,e^{-\e^{-\delta}} e^{s({\rm Im}\,\lambda_0 - \e^2 s^2- r - \omega)}
\ee

\noindent respectively. Following the previous computations, we may then replace \eqref{intro.intermediaire} by 
$$	
	e^{-\e^{-\delta}} e^{s ({\rm Im}\,\lambda_0 + \omega)} R\rho^{-1} \ll \,e^{-\e^{-\delta}}  e^{s({\rm Im}\,\lambda_0 - \e^2 s^2 - r - \omega)}
$$

\noindent and we finally get, instead of \eqref{intro.constraint}, the new constraint 
$$
	e^{\und{s}(\e^2 \und{s}^2 + r + \omega)} R\rho^{-1} \ll 1 .
$$

\noindent It can be fulfilled for any $\delta$ in $(0,2/3)$, which implies instability in Gevrey spaces $G^{\sigma}$ with $\sigma < 2/3$. We show in Sections \ref{section.assumptions} and \ref{section.ansatz} that assumptions of maximality and semi-simplicity for the most unstable eigenvalue lead to \eqref{intro.propa.bis.max} and \eqref{intro.lower.max}. These correspond to the assumptions of Theorem \ref{theorem.4}.

\section*{Notations}

\begin{itemize}

	\item For all $z\in\C^{m}$ and $k\in\N^{m}$, we put
	\be
		\label{notation.product}
		z^{k} = \prod_{i= 1,\ldots,m} z_i^{k_i} 
	\ee
	
	\item For all $k\in\N^{m}$
	\be
		\label{notation.binom}
		\binom{k_1 + \cdots + k_m}{k_1,\ldots,k_m} = \frac{\dsp (k_1+\cdots+k_m)!}{\dsp \prod_{i= 1,\ldots,m} k_i! } 
	\ee
	
	\item For all $m$ and $i\in\{1,\ldots,m\}$, we denote $1_{i}$ the $m$-uple with all coefficients null but the $i$-th:
	\be
		\label{notation.1i}
		1_{i} = (0,\ldots,0,1,0,\ldots,0)
	\ee
	
	\item For all reals $A$ and $B$ we note
	\be
		\label{notation.lesssim}
		A \lesssim B
	\ee
	
	\noindent if there is some constant independent of $\e$ such that
	$$ A \leq C B .$$
	
	\item For any functions $A$ and $B$ of $\e$, we denote
		\be
			\label{notation.ll}
			A \ll B \quad \Longleftrightarrow \quad A = o_{\e\to0}(B) .
		\ee
	
	\item For $r>0$ and $x_0\in\R^{d}$ we denote
	\be
		\label{notation.ball}
		B_{r}(x_0) = \left\{x\in\R^{d} \;\big|\; |x-x_0| < r\right\} .
	\ee
	
	%
	
\end{itemize}


\section{Main assumptions and results}
\label{section.assumptions}


\subsection{Definitions: Hölder well-posedness in Gevrey spaces}
\label{subsection.definitions}


We recall the definition of Gevrey functions on an open set $B$ of $\R^{d}$:

\begin{defi}[Gevrey functions]
	\label{def.gevrey}
	Let $\sigma\in(0,1)$. We define $G^{\sigma}(B)$ as the set of $C^{\infty}$ functions $f$ on $B$ such that, for all compact $K \subset B$ there are constants $C_K >0$ and $c_K>0$ that satisfy
	\be
		|\d^{\a} f|_{L^{\infty}(K)} \leq C_K c_K^{|\a|} |\a|!^{1/\sigma} \qquad \forall \a \in\N^{d}.
	\ee
	\noindent We then define a family of norms on $G^{\sigma}(B)$, for all compact $K\subset B$ and $c>0$ by
	\be
		||f||_{\sigma,c,K} = \sup_{\a} |\d^{\a} f|_{L^{\infty}(K)}c^{-|\a|}|\a|!^{-1/\sigma}.
	\ee
\end{defi}

For an introduction to Gevrey spaces and their properties, we refer to the book of Rodino \cite{rodino1993linear}. We introduce also space-time conical domains centered on $(0,x_0)\in\R\times\R^{d}$.

\begin{defi}[Conical domains]
	\label{defi.conical}
	For $x_0\in\R^{d}$, $R>0$, $\rho>0$ and $t\geq 0$ we define the set
	\be
		\label{defi.Omega.t}
		\Omega_{R,\rho,t}(x_0) = \left\{x \in\R^{d} \;\Big|\; R |x-x_0|_{1} + \rho t <1 \right\} 
	\ee

	\noindent with $|x|_{1} = \sum_{j=1,\ldots,d} |x_j|$ the $L^{1}$ norm on $\R^{d}$. Note that for all $t \geq \rho^{-1}$, $\Omega_{R,\rho,t}(x_0) = \emptyset$. We also denote
	\be
		\label{defi.Omega}
		\Omega_{R,\rho}(x_0) = \bigcup_{t\geq 0} \{t\}\times\Omega_{R,\rho,t}(x_0) = \left\{(t,x) \in\R\times\R^{d} \;\Big|\; 0\leq t < \rho^{-1} ,\; R |x-x_0|_{1} + \rho t <1 \right\} .
	\ee
\end{defi}

\noindent	Note that $\Omega_{R,\rho,t}$ is decreasing for the inclusion as a function of $R$, $\rho$ and $t$. In particular, $\Omega_{R,0,0}(x_0)$ is $B_{R^{-1}}(x_0)$.

The question is whether the Cauchy problem \eqref{Cauchy} is well-posed in Gevrey spaces or not, in the following sense

\begin{defi}[Hölder well-posedness]
	\label{defi.well_posedness}
	We say that {\rm \eqref{Cauchy}} is Hölder well-posed in $G^{\sigma}$ locally around $x_0\in\R^{d}$ if there are constants $r_0 > r_1 >0$, $c>0$, $C_{\text{in}}>0$, $C_{\text{fin}}$, $\rho>0$, $\a\in(0,1)$ such that for any $h$ in $G^{\sigma}(B_{r_0}(x_0))$ with 
	$$
		||h||_{\sigma,c,K} \leq C_{\text{in}} \qquad \forall K \text{ compact of } B_{r_0}(x_0)
	$$
	
	\noindent and all $R>r_1^{-1}$ the Cauchy problem {\rm \eqref{Cauchy}} associated to $h$ has a unique solution $u(t,x)$ in $C^{1}(\overline{\Omega}_{R,\rho}(x_0))$ with $|u|_{L^{2}(\Omega_{R,\rho}(x_0))} \leq C_{\text{fin}}$ and if moreover, given $h_1$ and $h_2$ in $G^{\sigma}(B_{r_0}(x_0))$ the corresponding solutions $u_1$ and $u_2$ satisfy the estimate for all $R>r_1^{-1}$ and $K$ compact subset of $B_{r_0}(x_0)$
	$$
		|u_1 - u_2|_{L^{2}(\Omega_{R,\rho}(x_0))} \lesssim ||h_1 - h_2||_{\sigma,c,K}^{\a}.
	$$
\end{defi}


\subsection{Assumptions}
\label{subsection.normal}

%
We define the principal symbol evaluated at a distinguished frequency $ \xi_0\in\R^{d}$ by
\be
	\label{def.A}
	A(t,x,u) = \sum_{j} A_j(t,x,u)\xi_{0,j} \quad , \; \forall (t,x,u) \in \R_{+}\times\R^{d}\times\R^{N}.
\ee

\begin{hypo}
	\label{hypo.1}

	We assume that for some $x_0\in\R^{d}$ and $\vec{u}_0\in\R^{N}$, the spectrum of $A(0,x_0,\vec{u}_0)$ is not real:
	\be
		\label{hypo.spectrum}
		{\rm Sp} A(0,x_0,\vec{u}_0) \not\subseteq \R .
	\ee

\end{hypo}


That is, the principal symbol $A$ is initially elliptic. 

\begin{notation}
	\label{notation.elli}
	We denote then
		\be
			\label{def.A_0}
			A_0 = A(0,x_0,\vec{u}_0)
		\ee

		\noindent which is a constant matrix with non-real spectrum by \eqref{hypo.spectrum}. Among the nonreal eigenvalues of $A_0$, we denote $\lambda_0$ the one with maximal positive imaginary part, denoted $\g_0$. We denote $\vec{e}_{+}$ the associated eigenvector. We denote also
		\be
			\label{def.undA}
			\und{A}(t,x) = A(t,x,\vec{u}_0).
		\ee

\end{notation}

Up to translations in $x$ and $u$, which do not affect our assumptions, and by homogeneity in $\xi$, we may assume 
\be
	\label{etc}
	x_0 = 0 \quad , \quad \vec{u}_0 = 0 \quad , \quad \xi_0 \in \mathbb{S}^{d-1} .
\ee

Under Assumption \ref{hypo.1} alone, we prove instability for the Cauchy problem \eqref{Cauchy} in some Gevrey indices (Theorem \ref{theorem.2} in Section \ref{subsection.result} below). We now formulate additional assumptions which yield instability for higher Gevrey spaces (Theorems \ref{theorem.3} and \ref{theorem.4} below).

\begin{hypo}
	\label{hypo.2}
	For some $x_0 \in \R^d$ and $\xi_0 \in \mathbb{S}^{d-1},$ the matrix $A_0$ has an eigenvalue $\lambda_0$ such that there holds $\lambda_0 \in \C \setminus \R,$ and ${\rm Im} \, \lambda_0 > {\rm Im} \, \mu,$ for any other eigenvalue $\mu$ of $A_0.$ Besides, the eigenvalue $\lambda_0$ is semisimple (which means algebraic and geometric multiplicities coincide) and belongs to a branch of semisimple eigenvalues of $A.$ Finally, $(0,x_0,\lambda_0)$ is not a coalescing point in the spectrum of $\underline A.$ 
	
	\end{hypo}
	
	\medskip
	
We denote $P_0$ the eigenprojector of $A_0$ associated with $\lambda_0,$ and $A_0^{-1}$ the partial inverse of $A_0,$ defined by $P_0 A_0^{-1} = 0,$ $A_0 A_0^{-1} = \Id - P_0.$ We also denote $(t,x) = (x_0, \dots, x_d),$ so that $\d_0 = \d_t,$ $\d_j = \d_{x_j}.$ 

\begin{remark}
	\label{remark.asumm.2}
	The non-coalescing assumption {\rm \ref{hypo.2}} implies (see {\rm \cite{kato2013perturbation}}, or Corollary {\rm 2.2} of {\rm \cite{texier2015rouche}}) that there is a smooth (actually, analytical) branch $\lambda$ of eigenvalues of $\underline A$ such that $\lambda(0,x_0) = \lambda_0.$ The corresponding local eigenprojector $\underline P$ is smooth as well. The local semisimplicity assumption means that $\underline A \,\underline P = \lambda \underline P,$ that is, in restriction to the eigenspace associated with $\lambda,$ the symbol $A$ is diagonal. A sufficient condition for semisimplicity is algebraic simplicity of the eigenvalue. 
\end{remark}


\begin{hypo}
	\label{hypo.3}
	With notation $P_0$ and $A_0^{-1}$ introduced just above Remark {\rm \ref{remark.asumm.2}},
	\begin{itemize}
		\item[(i)] there holds $P_0 \d_j \underline A(0,x_0) P_0 = 0,$ for all $j \in \{0, \dots, d\}.$ 
	\end{itemize}

	\noindent {\rm Under condition (i), the matrix
		\be
			\label{local.assum.2}
			P_0 \d_i \underline A A_0^{-1} \d_j \underline A P_0 + P_0 \d_j \underline A A_0^{-1} \d_i \underline A P_0  + P_0 \d_{ij}^2 \underline A P_0
		\ee
	\noindent (where derivatives of $\underline A$ are evaluated at $(0,x_0)$) has only non-zero eigenvalue (see {\rm \cite{kato2013perturbation}}, or Proposition {\rm 2.6} of {\rm \cite{texier2004short}}), which we denote $\mu_{ij}$.
	}
	\begin{itemize}
		\item[(ii)] The matrix $({\rm Im} \, \mu_{ij})_{0\leq i,j\leq d}$ is negative definite.
	\end{itemize}

\end{hypo}

\begin{remark}
	\label{remark.asum.3}
	Under Assumption {\rm \ref{hypo.2}}, Assumption {\rm \ref{hypo.3}} implies (see {\rm \cite{kato2013perturbation}}, or Proposition {\rm 2.6} of {\rm \cite{texier2004short}}) that the Hessian of ${\rm Im} \, \lambda$ at $(0,x_0)$ is negative definite, hence $(0,x_0)$ is a local maximum, in space-time, for ${\rm Im} \, \lambda.$
\end{remark}

%
%
%
%

\begin{hypo}
	\label{hypo.4}
	We assume that $f(t,x,u)$ is quadratic in $u$ locally around $u=\vec{u}_0$, that is
	\be
		\label{hypo.f}
		\d_{u} f(t,x,u) \big|_{u=\vec{u}_0=0} \equiv 0
	\ee
\end{hypo}
	


\subsection{Statement of the results}
\label{subsection.result}


	%
	%

In the statement below we use notations introduced in Definitions \ref{def.gevrey} and \ref{defi.conical}.

\begin{theo}
	\label{theorem.2}
	Under Assumptions {\rm \ref{hypo.1}} and {\rm \ref{hypo.4}}, the Cauchy problem \eqref{Cauchy} is not Hölder well-posed in Gevrey spaces $G^{\sigma}$ for all $\sigma \in(0,1/(m+1))$ where $m$ is the algebraic multiplicity of $\lambda_0$. That is for all $c>0$, $K$ compact of $\R^{d}$ and $\a\in(0,1]$, there are sequences $R_{\e}^{-1} \to 0$ and $\rho_{\e}^{-1} \to 0$, a family of initial conditions $h_{\e}\in G^{\sigma}$ and corresponding solutions $u_{\e}$ of the Cauchy problem on domains $\Omega_{R_{\e},\rho_{\e}}(x_0)$ such that
	\be
		\lim_{\e\to 0} ||u_{\e}||_{L^2(\Omega_{R_{\e},\rho_{\e}}(x_0))} / ||h_{\e}||_{\sigma,c,K}^{\a} = +\infty .
	\ee
	
	\noindent The time of existence of the solutions $u_{\e}$ is at least of order $\e^{1-\sigma} $.
\end{theo}

We prove the instability for a larger band of Gevrey indices under stronger assumptions. First, the semisimplicity and non-coalescing  Assumption \ref{hypo.2} allows for a critical index equal to $1/2$:

\begin{theo}
	\label{theorem.3}
	Under Assumptions {\rm \ref{hypo.2}} and {\rm \ref{hypo.4}}, the result of Theorem {\rm \ref{theorem.2}} holds for any Gevrey index $\sigma$ in $(0,1/2)$.
\end{theo}

Second, under Assumption \ref{hypo.2}, the null condition (i) and the sign condition (ii) in Assumption {\rm \ref{hypo.3}} allow for the critical index to go from $1/2$ up to $2/3$:

\begin{theo}
	\label{theorem.4}
	Under Assumptions {\rm \ref{hypo.2}}, {\rm \ref{hypo.3}} and {\rm \ref{hypo.4}}, the result of Theorem {\rm \ref{theorem.2}} holds for any Gevrey index $\sigma$ in $(0,2/3)$.
\end{theo}

The rest of the paper is devoted to the proof of Theorems \ref{theorem.2}, \ref{theorem.3} and \ref{theorem.4}.

\begin{remark}
	\label{remark.aprestheo3}
	
	Higher-order null and sign conditions allow for a greater critical index. Precisely, under Assumption {\rm \ref{hypo.2}}, if $(0,x_0)$ is a local maximum for $ {\rm Im}\lambda $, and if there holds $ \lambda(\e s,x_0) - \lambda(0,x_0) = O(\e s)^{2k-1} $, then our proof implies ill-posedness with a critical Gevrey index equal to $2k/(2k+1)$. These null and sign conditions can be expressed in terms of derivatives of $\und{A}$, the partial inverse $A_0^{-1}$ and the projector $P_0$, see {\rm \cite{kato2013perturbation}}, or Remark {\rm 2.7} of {\rm \cite{texier2004short}}. See also Remark {\rm \ref{remark.amelioration}}.


	%
	

\end{remark}

	%
%


\section{Highly oscillating solutions and reduction to a fixed point equation}
\label{section.ansatz}


\subsection{Preparation of the equation}


We want to compare two solutions of \eqref{Cauchy} with initial data $h_1$ and $h_2$ satisfying both
$$
	h_i(x=0) = 0 \qquad \text{for } i=1,2
$$ 

\noindent to fit with $\vec{u}_0 = 0$ in \eqref{etc}. We can choose $h_1$ analytic, which lead by Cauchy-Kovalevskaya theorem to an analytic solution $u_1$ in some small neighborhood of $(0,0)\in\R_{t}\times\R_{x}^{d}$. Then changing $u$ into $u-u_1$ in \eqref{Cauchy} we get a new Cauchy problem
\be
	\label{Cauchy.bis}
	\d_{t}u = \sum_{j}A_{j}(t,x,u)\d_{x_j}{u} + F(t,x,u)u \; , \quad u(0,x) = h(x)
\ee

\noindent with $F(t,x,u)\in\R^{N\times N}$ is also analytic, by analyticity of $f$ and $u_1$. We consider for $h$ small analytical functions satisfying $h_{|x=0}=0$, as perturbations of the trivial datum $h\equiv 0$.


\subsection{Highly oscillating solutions}


As in \cite{metivier2005remarks} we look for high oscillating solutions of \eqref{Cauchy.bis} with the aim of seeing the expected growth. In this view we posit the following ansatz
\be
	\label{ansatz}
	u_{\e}(t,x) = \e \u(t/\e,x,x\cdot\xi/\e) 
\ee

\noindent where the function $\u(s,x,\theta)$ is $2\pi$-periodic in $\theta$. We introduce for any analytical function $H(t,x,u)$ the compact notation
\be
	\label{def.mathbf.H}
	\H(s,x,\u) = H\left( \e s,x,\e\u\right) .
\ee

For $u_{\e}(t,x)$ to be solution of \eqref{Cauchy.bis} it is then sufficient that $\u(s,x,\theta)$ solves the following equation
\be
	\label{ds.u.passager}
	\d_{s} \u = {\mathbf A}\,\d_{\theta} \u + \e\left(\sum_j {\mathbf A}_j \d_{x_j} \u + {\mathbf F} \, \u\right) 
\ee

\noindent where we use the notation \eqref{def.mathbf.H} for the $\mathbf{A}_j$ and $\mathbf{F}$, and $A$ is defined by \eqref{def.A}.

As we focus our study in a neighborhood of the distinguished point $(0,0)\in\R_{t}\times\R^{d}_{x}$ (recall that $x_0 = 0$), we rewrite now \eqref{ds.u.passager} as
\be
	\label{equation.u.compact}
	\d_{s} \u - \und{{\mathbf A}} \d_{\theta} \u = {\mathbf G}(s,x,\u)
\ee

\noindent where $\und{{\mathbf A}}(s,x) = \und{A}(\e s ,x)$ in accordance with notation \eqref{def.mathbf.H}. We define the source term
\be
	\label{def.G}
	{\mathbf G} = \left({\mathbf A} - \und{{\mathbf A}} \right)\,\d_{\theta} \u + \e\left(\sum_j {\mathbf A}_j \d_{x_j} \u + {\mathbf F} \, \u\right)  
\ee

\noindent using the notation \eqref{def.mathbf.H}.


\subsection{Upper bounds for the propagator}


To solve the Cauchy problem of the equation \eqref{equation.u.compact} with initial datum $h_{\e}$ specified in Section \ref{subsection.free.solution}, we first study the case ${\mathbf G} \equiv 0$, that is
\be
	\label{equation.free}
	\d_{s} \u(s,x,\theta) - \und{{\mathbf A}}( s,x) \d_{\theta} \u(s,x,\theta) = 0 .
\ee

\noindent Note that this equation is linear, non autonomous and non scalar. We define the matrix propagator $U(s',s,x,\theta)$ as the solution of
\be
	\label{equation.propagator}
	\d_{s} U(s',s,x,\theta) - \und{{\mathbf A}}( s,x) \d_{\theta} U(s',s,x,\theta) = 0 \; , \quad U(s',s',x,\theta) \equiv {\rm Id} .
\ee

\noindent and $U(s',s,x,\theta)$ is periodic in $\theta$, following the ansatz \eqref{ansatz}. 


\begin{lemma}[Growth of the propagator]
	\label{lemma.growth.propa}
	The matrix propagator $U(s',s,x,\theta)$ satisfies the following growth of its Fourier modes in the $\theta$ variable:
	\be
		\label{growth.propa.n}
		|U_n(s',s,x)| \lesssim \omega^{-(m-1)}\,\exp\left( \int_{s'}^{s} \g^{\sharp}(\tau \,;R,\omega) d\tau \, |n|\right) \quad , \quad \forall \, 0 \leq s' \leq s \text{  and  } \forall \, n\in\Z .
	\ee
	
	\begin{itemize}
		\item Under Assumption {\rm \ref{hypo.1}}, bound \eqref{growth.propa.n} holds with
		\be
			\label{bound.g.sharp}
			\g^{\sharp}(\tau \,;R,\omega) = \g_0 + \e \tau + R^{-1} + \omega 
		\ee
	
		\noindent where $\g_0$ is defined in Notation {\rm \ref{notation.elli}}, $m \geq 1$ is the algebraic multiplicity of $\lambda_0$. The bounds hold for $\omega >0$ small enough, uniformly in $x$ in the ball $B_{R^{-1}}(0)$. 
		
		\item Under Assumption {\rm \ref{hypo.2}}, bound \eqref{growth.propa.n} holds with $ m =1$ and 
		\be
			\label{bound.g.sharp.bis}
			\g^{\sharp}(\tau \,;R,\omega) = \g_0 + \e \tau + R^{-1}
		\ee
		
		\noindent with $\omega = 0$, both uniformly in $x$ in the ball $B_{R^{-1}}(0)$.
	
		\item Under Assumptions {\rm \ref{hypo.2}} and {\rm \ref{hypo.3}}, bound \eqref{growth.propa.n} holds with $\omega = 1$ and 
		\be
			\label{bound.g.sharp.max}
			\g^{\sharp}(\tau \,;R,\omega) = \g_0
		\ee
		
		\noindent The bounds hold uniformly in $x$ in the ball $B_{R^{-1}}(0)$. 
	
	\end{itemize}
	
\end{lemma}

\noindent In the framework of Assumption \ref{hypo.1}, the parameter $\omega$ is chosen in Proposition \ref{prop.below}.

\begin{proof}
	
	As $\und{A}(t,x)$ does not depend on $\theta$, equation \eqref{equation.propagator} reads in Fourier transform in $\theta$ as
	$$
		\d_{s}U_n(s',s,x) - in \und{A}(\e s , x) U_n(s',s,x) \; , \quad U_n(s',s,x) = {\rm Id}
	$$
	
	\noindent where $U_n$ is the $n$-th Fourier component of $U(\theta)$. That implies that operator $U(\theta)$ acts diagonally on each Fourier components.
	
	The bounds \eqref{growth.propa.n} - \eqref{bound.g.sharp} follow from elementary, and purely linear-algebraic, arguments detailed in Sections (4.2) and (4.3) of \cite{lerner2015onset}. 
	
	The bounds \eqref{growth.propa.n} - \eqref{bound.g.sharp.bis} follow from a smooth partial diagonalization of symbol $\und{A}$ over the eigenspace associated with $\lambda$. In particular, there is no diagonalization or trigonalization error, hence $m = 1$ in \eqref{growth.propa.n} and $\omega = 0$ in \eqref{bound.g.sharp.bis}.
	
	The bounds \eqref{growth.propa.n} - \eqref{bound.g.sharp.max} follow from a smooth partial diagonalization as described above, and the fact that the imaginary part of $\lambda$ is maximal at $(t,x) = (0,x_0),$ as described in Remark \ref{remark.asum.3}.
	
\end{proof}

\subsection{Free solutions}
\label{subsection.free.solution}


After getting the previous upper bounds for the propagator, we seek initial conditions $h_{\e}$ that achieve the maximal growth. For this purpose, following again \cite{metivier2005remarks} we introduce the following high-oscillating, small and well-polarized initial data
\be
	\label{def.initial.data}
	h_{\e}(x) = \e\,e^{-M(\e)} {\rm Re}\left(e^{-ix\cdot\xi_0/\e} \vec{e}_{+} + e^{ix\cdot\xi_0/\e} \vec{e}_{-}\right)
\ee

\noindent which correspond in the ansatz \eqref{ansatz} of high-oscillating solutions to
\be
	\label{def.initial.data.ansatz}
	\h_{\e}(x,\theta) = e^{-M(\e)} {\rm Re}\left(e^{-i\theta} \vec{e}_{+} + e^{i\theta} \vec{e}_{-}\right) .
\ee

\noindent Here $\vec{e}_{+}$ is defined in Notation \ref{notation.elli}, and $\vec{e}_{-} = \overline{\vec{e}_{+}}$. The parameter $M(\e)$ is large in the limit $\e\to 0$, chosen such that the Gevrey norm of $h_{\e}$ is small. We introduce also 

\be
	\label{free.solution}
	\f_{\e}(s,x,\theta) = U(0,s,x,\theta) \h_{\e}(x,\theta)
\ee

\noindent which we call the free solution of equation \eqref{equation.u.compact} as it solves the equation for ${\mathbf G} \equiv 0$.

\subsubsection{Growth of the free solution}

\begin{lemma}[Growth of the free solution]
	\label{lemma.growth.free.solution}
	There holds
	\be
		\label{growth.free.solution}
		|\f_{\e}(s,x,\theta)| \gtrsim \omega^{-(m-1)}\,e^{-M(\e)}\exp\left( \int_{0}^{s} \g^{\flat}(\tau \,;r,\omega) d\tau  \right).
	\ee
	
	\begin{itemize} 
		\item Under Assumption {\rm \ref{hypo.1}}, bound \eqref{growth.free.solution} holds with
			\be
				\label{bound.g.flat}
				\g^{\flat}(\tau \,; r,\omega) = \g_0 - \e \tau - r - \omega ,
			\ee
	
			\noindent pointwise in $(s,x,\theta)\in[0,\und{s})\times B_r(x_0)\times\mathbb{T}$. 
		
		\item Under Assumption {\rm \ref{hypo.2}}, bound \eqref{growth.free.solution} holds with $m=1$ and
			\be
				\label{bound.g.flat.bis}
				\g^{\flat}(\tau \,; r,\omega) = \g_0 - \e \tau - r ,
			\ee
	
			\noindent with $\omega = 0$, pointwise in $(s,x,\theta)\in[0,\und{s})\times B_r(0)\times\mathbb{T}$. 
	
		\item Under Assumptions {\rm \ref{hypo.2}} and {\rm \ref{hypo.3}}, bound \eqref{growth.free.solution} holds with $\omega = 1$ and
			\be
				\label{bound.g.flat.max}
				\g^{\flat}(\tau \,; r,\omega) = {\rm Im}\,\lambda(\e\tau,0) - r .
			\ee
			
			\noindent pointwise in $(s,x,\theta)\in[0,\und{s})\times B_r(0)\times\mathbb{T}$. 
	
	\end{itemize}
\end{lemma}

\begin{proof}
	
	Our choice of datum \eqref{def.initial.data}-\eqref{free.solution} allows an exact localization at the distinguished frequency $\xi_0$. Similarly to the proof of Lemma \ref{lemma.growth.propa}, the lower bounds follow from linear algebraic arguments detailed in \cite{lerner2015onset}.

\end{proof}

\subsubsection{Smallness of the free solution and Gevrey index}

The size of the Gevrey-$\sigma$ norm of the initial data $h_{\e}$ is linked to the exponent $M(\e)$ as shown by the following

\begin{lemma}
	\label{size.gevrey.exp}
	For any $\sigma\in(0,1)$, $c>0$ and $K$ a compact of $\R^{d}$ there holds
	\be
		\label{esti.gevrey.free}
		 \dsp ||h_{\e}||_{\sigma,c,K} \lesssim \e\,\exp\left( - M(\e) + \frac{\e^{-\sigma}}{\sigma c^{\sigma}}\right) .
	\ee
	
	\noindent We emphasize that the constant in the previous inequality does not depend on $K$.
\end{lemma}

\begin{proof}
	First we have
	$$ 
		\dsp \d_{x}^{k} e^{\pm ix\cdot\xi_{0}/\e} = \left(\pm i\xi_{0}/\e\right)^{k}e^{\pm ix\cdot\xi_{0}/\e} \; ,\quad \forall k\in\N^{d} \;,\quad \forall x\in\R^{d} 
	$$
	
	\noindent using notation \eqref{notation.product} and then
	$$ 
		\dsp |\d_{x}^{k} e^{\pm ix\cdot\xi_{0}/\e}| \leq C_d \, \e^{-|k|} \; ,\quad \forall k\in\N^{d} \;,\quad \forall x\in\R^{d} 
	$$

	\noindent as $|\xi_{0}| = 1$, with $C_d>0$ a constant depending only of the dimension $d$. So that for any compact $K$ of $\R^{d}$ and by definition \eqref{def.initial.data} of the initial data $h_{\e}$, there holds
	$$ 
		\dsp c^{-|k|}|k|!^{-1/\sigma} |\d_{x}^{k} h_{\e} |_{L^{\infty}(K)} \lesssim \e\, e^{-M(\e)} \,\e^{-|k|}\,c^{-|k|} |k|!^{-1/\sigma} \;,\quad \forall k\in\N^{d} .
	$$

	\noindent By Definition \ref{def.gevrey} of the Gevrey norms, this implies
	$$
		\dsp ||h_{\e}||_{\sigma,c,K} \lesssim \e e^{-M(\e)} \, \sup_{k\in\N^{d}} \e^{-|k|}\,c^{-|k|} |k|!^{-1/\sigma} .
	$$
	
	\noindent For any $t>0$ we have
	$$
		\frac{t^{|k|}}{|k|!} \leq e^{t} \quad , \quad \forall t>0 \;, \quad \forall k\in\N^{d} 
	$$
	
	\noindent and note that the loss is smaller as $|k|$ is larger. This leads to
	$$
		||h_{\e}||_{\sigma,c,K} \lesssim \e e^{-M(\e)} \, \sup_{k\in\N^{d}} \e^{-|k|}\,c^{-|k|} \left(t^{|k|}e^{-t}\right)^{-1/\sigma}
	$$

	\noindent and then by putting $t = \e^{-\sigma}c^{-\sigma}$ into this last inequality, we finally obtain the inequality \eqref{esti.gevrey.free}.
	
\end{proof}

\noindent As we need $h_{\e}$ to be small both in Gevrey-$\sigma$ norm and in amplitude, we posit
\be
	\label{size.gevrey}
	M(\e) = \e^{-\delta} , \quad \delta \in(\sigma,1) .
\ee

\begin{remark}
	\label{remark.size.gevrey}
	With the previous definition \eqref{size.gevrey}, the initial data $h_{\e}$ is exponentially small, both in Gevrey-$\sigma$ norm and in absolute value. This last point is of importance, as we need $h_{\e}$ to be small enough to see the exponential growth of the solution it generates in a sufficiently long time $T(\e)$ to be defined later. A constraint on this final time will lead to a constraint on the size $e^{-M(\e)}$ of $h_{\e}$, and then to the constraint $\sigma < \delta$ (see \eqref{size.gevrey}) bearing on the admissible Gevrey regularity.
\end{remark}


\subsection{Fixed point equation}


Using the propagator $U(s',s,\theta)$, the free solution \eqref{def.initial.data} and the Duhamel formula, we can express now \eqref{equation.u.compact} as the fixed point equation
\be
	\label{fixed.point.equation}
	\u(s,x,\theta) = \f_{\e}(s,x,\theta) + \int_{0}^{s} U(s',s,x,\theta) {\mathbf G}(s',\u(s',x,\theta)) ds'
\ee

\noindent where ${\mathbf G}(\u)$ is defined by \eqref{def.G}. We denote the integral term
\be
	\label{defi.T}
	T(s,\u) = \int_{0}^{s} U(s',s) {\mathbf G}(s',\u(s')) ds'  
\ee

\noindent which we split into three parts thanks to definition \eqref{def.G} like
\begin{eqnarray}
	T(s,\u) & = & \int_{0}^{s} U(s',s) \left[ \left( {\mathbf A} - \und{{\mathbf A}} \right)\,\d_{\theta} \u + \e\left(\sum_j {\mathbf A}_j \d_{x_j} \u + {\mathbf F} \, \u\right)  \right] ds' \nonumber \\
	& = & T^{[\theta]}(s,\u) + T^{[x]}(s,\u) + T^{[\u]}(s,\u) \label{defi.T.decoupage}
\end{eqnarray}

\noindent where we define
\begin{eqnarray}
	T^{[\theta]}(s,\u) & = & \int_{0}^{s} U(s',s) \,\left( {\mathbf A} - \und{{\mathbf A}} \right) \,\d_{\theta} \u(s') ds' \label{def.T_theta}\\
	T^{[x]}(s,\u) & = & \int_{0}^{s} U(s',s)\, \sum_j\left(\e\A_j(s',\u(s'))\right) \d_{x_j} \u(s') ds' \label{def.T_x} \\
	T^{[\u]}(s,\u) & = & \int_{0}^{s} U(s',s) \, \left(\e \F(s',\u(s'))\right) \, \u(s') ds' \label{def.T_u} .
\end{eqnarray}



\subsection{Sketch of the proof}
\label{subsection.sketch}


We have now reduced the initial question of finding a family of initial data $h_{\e}$ generating a family of appropriately growing analytic solutions $u_{\e}$ to the fixed point equation \eqref{fixed.point.equation} for operator $T$. To find smooth solutions of this equation we have first to find a suitable functional space $\EE$ with the following properties:

\begin{itemize}

	\item The space $\EE$ should be a Banach space to make use of the Banach fixed point theorem. Moreover functions of $\EE$ should be smooth functions in variables $(s,x,\theta)$.
	
	\item The space $\EE$ should be a Banach algebra equipped with norm $|||\cdot|||$ satisfying $ |||\u \v||| \leq |||\u|||\,|||\v||| $ as we deal with non linear terms ${\mathbf G}(\u)$.
	
	\item We will need to precisely evaluate the action of derivation operators $\d_{x_j}$ and $\d_{\theta}$ on $\EE$. In an analytical framework, these are \textit{a priori} not bounded operators, and as in \cite{ukai2001boltzmann} and \cite{metivier2005remarks} we should use time integration to get back boundedness in $\EE$ with some loss in the bounds we should quantify.
	
	\item The space $\EE$ should be invariant by the flow $U(s',s,x,\theta)$. In this view, we need estimates in $\EE$ for the matrix flow $U_n(s',s,x)$. 
	
	\item The operator $T$ should be a contraction on $\EE$ for well chosen parameters, and for small $\e$. 
\end{itemize}

\noindent To this end, Section \ref{section.spaces} will present the satisfying functional setting, and Section \ref{section.contraction} will prove the contraction estimate for $T$.

In order to prove the Hadamard instability, the existence of solutions to the fixed point equations \eqref{fixed.point.equation} is not sufficient. The key of the proof is to obtain for the solution $\u$ associated to $\f_{\e}$ the same kind of growth as $\f_{\e}$, as developed in Section \ref{intro.endgame}, and this is the aim of Section \ref{section.existence}. Finally, such a growth for $\u$ leads to the Hadamard instability of the Cauchy problem \eqref{Cauchy.bis}. This completes the proof of Theorems \ref{theorem.2}, \ref{theorem.3} and \ref{theorem.4} in Section \ref{section.Hadamard}.


\section{Majoring series and functional spaces}
\label{section.spaces}


\subsection{Properties of majoring series}
\label{subsection.majoring}


One aim of the paper is to construct a family of analytical solutions of the fixed point equation \eqref{fixed.point.equation}. We deal with functions of several variables: $x$, $(s,x)$ or $(s,x,u)$, and the question of analyticity of these functions with respect to all variables or only to some arises. In that purpose we consider formal series of $\mu$ variables, with complex coefficients that depend eventually on a parameter $y$ in some open domain $\mathcal{O}$ of $\C^{\mu'}$. We denote such formal series 
$$
	\phi(z,y) = \sum_{k\in\N^{\mu}} \phi_{k}(y) z^{k} \quad , \quad \phi_{k}(y) \in \C \, , \quad \forall \, k\in\N^{\mu}\, , \; \forall \, y\in\mathcal{O} 
$$

\noindent where we introduce formal unknowns $z = (z_1,\ldots,z_{\mu})$.  A formal series $\phi(z,y)$ is really a $y$-dependent
sequence $\left(\phi_k(y)\right)_k$ indexed by $k\in\N^{\mu}$ . An important parameter is the dimension $\mu$ of the indices $k$. We define now the relation of majoring series between two formal series $\phi(z,y)$ and $\psi(Z,y)$, with $z$ and $Z$ denoting $\mu$ variables.

\begin{defi}[Majoring series]
	\label{definition.maj}
	For $\phi(z,y)$ and $\psi(Z,y)$ formal series of respectively variable $z$ and variable $Z$, and $y$ a parameter in some open domain $\mathcal{O}$ of $\C^{\mu'}$, with furthermore
	$$ \psi(Z,y) = \sum_{k\in\N^{\mu}} \psi_k(y) Z^{k} \quad \text{with} \quad \psi_k(y) \geq 0 \quad \forall k\in\N^{\mu} \,,\; \forall y\in\mathcal{O} $$
	\noindent we define 
	\be
		\label{def.majoring.series}
		\phi(z,y) \prec_{y} \psi(Z,y) \quad \Longleftrightarrow \quad \Big( \forall k\in\N^{\mu}\, ,\; \forall y\in\mathcal{O} : \quad |\phi_k(y)| \leq \psi_k(y) \Big)
	\ee
\end{defi}

\begin{remark}
	In notation $\prec_{y}$ we emphasize that we consider $y$ as a parameter in the formal series $\phi(z,y)$.
\end{remark}

\noindent In the following we sum up several classical properties of the relation \eqref{def.majoring.series} (see \cite{cartan1995theorie}).

\begin{lemma}
	\label{majoring.properties}
	Let $\phi$ and $\psi$ be as in the previous definition, with $\phi \prec_{y} \psi$. Then
	
	\begin{enumerate}
	
		\item If $\psi$ converges at a point $(Z,y)$ with $Z_i \geq 0$ for all $i=1,\ldots,m$, then $\phi$ converges on all $(z,y)$ such that $|z_i|\leq Z_i$, and 
			\be
				|\phi(z_1,\ldots,z_{\mu},y)| \leq \psi(|z_1|,\ldots,|z_{\mu}|,y) 
			\ee
			
		\item The relation $\prec_{y}$ is compatible with formal derivations: denoting $\d_{i}$ the formal derivation along the $i$-th variable, we have
			\be 
				\phi \prec_{y} \psi \quad \Longrightarrow \quad \d_{i}\phi(z,y) \prec_{y} \d_{i}\psi(Z,y) 
			\ee
			
		\item The relation $\prec_{y}$ is compatible with multiplication:
			\be
				\phi_1 \prec_{y} \psi_1 \text{ and } \phi_2 \prec_{y} \psi_2 \quad \Longrightarrow \quad \phi_1 \phi_2 \prec_{y} \psi_1\psi_2 
			\ee

		\item There is a constant $c_0>0$ such that the series 
			\be
				\label{defi.Phi}
				\Phi(z_1) =  \sum_{k \geq 0} \frac{\dsp c_0}{\dsp k^2+1} z_1^{k}
			\ee
			\noindent satisfies 
			\be
				\label{Phi2}
				\Phi^2 \prec \Phi
			\ee
			\noindent The series $\Phi$ is analytic on $B_{1}(0)$, defined in \eqref{notation.ball}.
	\end{enumerate}
\end{lemma}

\begin{proof}
We give here a short proof of this Lemma.
	\begin{enumerate}
	
		\item Assume that $\psi(Z,y)$ is converging at a point $(Z,y)$, with all $Z_i \geq 0$. By definition of the majoring series, we have for all $k\in\N^{\mu}$ the inequality $ |\phi_{k}(y)| \leq \psi_{k}(y) $. Since the series $\sum_{k} \psi_{k}(y)Z^{k}$ is convergent, then for all $z\in\C^{\mu}$ such that $|z_i| \leq Z_i$ the series $\sum_{k} \phi_{k}(y) z^{k}$ converges and there holds by \eqref{notation.product} and Definition \ref{definition.maj}
		\begin{eqnarray*}
			\left| \sum_{k\in\N^{\mu}} \phi_{k}(y) z^{k} \right| & \leq & \sum_{k\in\N^{\mu}} |\phi_{k}(y)| \prod |z_j|^{k_j} \\
			  & \leq & \sum_{k\in\N^{\mu}} \psi_{k}(y) \prod |z_j|^{k_j} \\
			  & = & \psi(|z_1|,\ldots,|z_{\mu}|,y) 
		\end{eqnarray*}
		
		\noindent Hence the importance of using two different notations for the $\mu$ variables, $z$ and $Z$.
		
		\item By definition of formal derivation $\d_{i}$, there holds
		$$ \d_{i}\phi(z,y) = \sum_{k\in\N^{\mu}} (k_i+1) \phi_{k+1_{i}}(y) z^{k} $$
		
		\noindent where $1_{i}$ is defined by \eqref{notation.1i} and for all $k\in\N^{\mu}$ there holds
		$$ |(k_i+1) \phi_{k+1_{i}}(y)| \leq (k_i+1) \psi_{k+1_{i}}(y) $$
		
		\noindent by Definition \ref{definition.maj}, which is exactly the $k$-th coefficient of the formal series $\d_{i}\psi(Z,y)$.
		
		\item Let $\phi^1$, $\phi^2$, $\psi^1$ and $\psi^2$ be such that $ \phi^1 \prec_{y} \psi^1$ and $ \phi^2 \prec_{y} \psi^2 $. By definition of the multiplication of two formal series, the coefficients of the formal series $\phi^1\phi^2(z,y)$ in $z$ are
		$$ (\phi^1\phi^2)_{k}(y) = \sum_{p=0}^{k} \phi^1_{p}(y)\phi^2_{k-p}(y) $$
		
		\noindent and then for all $y\in\mathcal{O}$ and $k\in\N^{\mu}$ there holds
		\begin{eqnarray*}
			\left|(\phi^1\phi^2)_{k}(y)\right| & \leq & \sum_{p=0}^{k} \left|\phi^1_{p}(y)\right| \,\left|\phi^2_{k-p}(y)\right| \\
				& \leq & \sum_{p=0}^{k} \psi^1_{p}(y) \,\psi^2_{k-p}(y)
		\end{eqnarray*}
		
		\noindent because $\phi^{1} \prec_{y} \psi^{1}$ and $\phi^2 \prec_y \psi^2$. As the right-hand side of the previous inequality is just $(\psi^1\psi^2)_k(y)$, this ends the proof.
		
		\item For $\mu=1$ and $\mu'=0$, we consider the series
		$$ \Phi(z) = \sum_{k\in\N} \frac{\dsp c_0}{\dsp k^2+1} z^{k} $$
		
		\noindent We compute
		$$ \Phi^2(z) = \sum_{k\in\N} \sum_{p=0}^{k} \frac{\dsp c_0}{\dsp p^2+1}\frac{\dsp c_0}{\dsp (k-p)^2+1} z^{k} . $$
		
		\noindent To prove the existence of some $c_0 >0$ such that \eqref{Phi2} holds, it suffices to prove that
		$$ \sum_{p=0}^{k} \frac{\dsp k^2+1}{\dsp (p^2+1)((k-p)^2+1) } $$
		
		\noindent is bounded for all $k\in\N$. Thanks to $k^2 \leq 2(p^2 + (k-p)^2) $ there holds
		$$ \sum_{p=0}^{k} \frac{\dsp k^2+1}{\dsp (p^2+1)((k-p)^2+1) } \leq 4 \sum_{p=0}^{k} \frac{\dsp 1}{\dsp p^2+1} \leq  4\sum_{p\in\N} \frac{\dsp 1}{\dsp p^2+1} $$ 
		
		\noindent which suffices to end the proof.

	\end{enumerate}

\end{proof}

After these abstract considerations we come back to series in the spatial variable $x$, where $t$ a parameter. The principle behind the relation of majoring series is to replace unknown analytical functions by a fixed, well-known series. In this view we consider the series in $d$ variables $(X_1, \ldots , X_d)$, with $t\in[0,\rho^{-1})$ a parameter and $R$ and $\rho$ some positive constants
\be
	\label{def.Phi}
	\Phi(RX_1+\cdots+RX_d+\rho t) = \sum_{k\in\N^{d}} \left(R^{|k|}\sum_{p\in\N} \frac{c_0}{(|k|+p)^2+1} \binom{|k|+p}{k,p} \rho^p t^p \right) X^{k} 
\ee

\noindent using the notations \eqref{notation.product} for $X^{k}$ and \eqref{notation.binom} for $\binom{|k|+p}{k,p}$. We denote
\be
	\label{def.Phi_k.t}
	\Phi_{k}(t) = R^{|k|}\sum_{p\in\N} \frac{c_0}{(|k|+p)^2+1} \binom{|k|+p}{k,p} \rho^p t^p \; , \quad \forall k\in\N^{d}
\ee

\noindent where it is implicit that $\Phi_{k}(t)$ depend also on $R$ and $\rho$. Note that the series in the right hand side of \eqref{def.Phi_k.t} is convergent for $|t| <\rho^{-1}$. Since the series $\Phi(z)$ converges in $B_1(z=0)$, the series $\Phi(RX_1+\cdots+RX_d+\rho t)$ is convergent as a series in $X$ and $t$ variables on $\Omega_{R,\rho}(0)$ defined by \eqref{defi.Omega}.

From now on, we will note for convenience and with an abuse of notation 
\be
	\Phi(RX+\rho t) = \Phi(RX_1+\cdots+RX_d+\rho t)
\ee

\noindent as the reference series in the $x$ variable, for some positive constants $R$ and $\rho$. In the following Lemma we sum up properties for formal series $\phi$ in $d$ variables with one parameter $t$ that satisfy 
$$\phi(x,t) \prec_{t} C \Phi(R X + \rho t)$$

\noindent for some $C>0$. This is equivalent, thanks to \eqref{def.majoring.series}, \eqref{def.Phi} and \eqref{def.Phi_k.t} to
\be
	\label{def.maj}
	|\phi_{k}(t)| \leq C \Phi_{k}(t)  \quad , \quad \forall k \in\N^{d} \, \text{ and } \; 0\leq t < \rho^{-1} .
\ee

\begin{lemma}
	\label{lemma.phi.prec.Phi}
	For $\phi(x,t)$ a formal series in $x$ with $ \phi(x,t) \prec_{t} C\Phi(R X + \rho t) $ there holds
		
	\begin{enumerate}
	
	\item $\phi(x,t)$ is analytic as a series in $x$ in the domain $\Omega_{R,\rho,t}(0)$ for all $0\leq t < \rho^{-1}$. 
	
	\item For all $0\leq t < \rho^{-1}$, there holds
		\be
			\label{dx.Phi}
			\d_{x_j} \phi(x,t) \prec_{t} CR \Phi'(R X + \rho t)
		\ee
	\noindent with $\Phi'$ the derivative of $\Phi$.
	
	\item For any $R \geq R_0$ and $\rho \geq \rho_0$, there holds
	\be
		\label{R_0.R}
		\Phi(R_0X+\rho_0 t) \prec_{R_0,\rho_0,R,\rho,t} \Phi(RX+\rho t) .
	\ee
	
	\item For any $R>0$, $\rho>0$ and $0 \leq t' < t < \rho^{-1}$, there holds
	\be
	\label{t'.t}
		\Phi(RX + \rho t') \prec_{t',t} \Phi(RX + \rho t) .
	\ee
	
	\end{enumerate}
\end{lemma}

\begin{proof}

	\begin{enumerate}
	
		\item By the first property of Lemma \ref{majoring.properties}, the formal series $\phi(x,t)$ is analytic in $x$ on the domain of convergence of the series $\Phi(RX + \rho t)$ thought as a series in $X$ variable. As it is just $\Omega_{R,\rho,t}(0)$, defined by \eqref{defi.Omega.t}, the function $\phi(x,t)$ is analytic on $\Omega_{R,\rho,t}(0)$ as a series in the $x$ variable for all $0\leq t < \rho^{-1}$.
		
		\item By the second property of Lemma \ref{majoring.properties} there holds $ \d_{x_j}\phi(x,t) \prec_{t} C\d_{X_j} \left(\Phi(RX+\rho t)\right) $ and as
		$$ \d_{X_j} \left(\Phi(RX+\rho t)\right) = \d_{X_j} \left(\Phi(RX_1 + \cdots + RX_{d} +\rho t)\right) = R\Phi'(RX+\rho t) $$
		
		\noindent for all $0\leq t < \rho^{-1}$, we finally get \eqref{dx.Phi}.
		
		\item Thanks to notation \eqref{def.Phi_k.t} we have $\Phi(RX + \rho t) = \sum_{k\in\N^{d}} \Phi_{k}(t) X^{k} $ for all $0 \leq t < \rho^{-1}$, where we recall it is implicit that the coefficients $\Phi_{k}(t) = \Phi_{k}(t,R,\rho)$ depend also on $R$ and $\rho$. In the definition \eqref{def.Phi_k.t} we easily see that
		$$ 
			\Phi_{k}(t,R_0,\rho_0) \leq \Phi_{k}(t,R,\rho) \quad , \quad \forall R\geq R_0,\; \forall \rho \geq \rho_0 ,\; \forall 0 \leq t < \rho^{-1} 
		$$
		
		\noindent which is exactly \eqref{R_0.R}.
		
		\item In the same way we see that, $ R$ and $\rho$ being fixed, the coefficients $\Phi_{k}(t)$ are increasing functions of $t$:
		$$ \Phi_{k}(t') \leq \Phi_{k}(t) \qquad \forall k\in\N^{d} , \; \forall 0 \leq t' < t < \rho^{-1} $$
		\noindent which is exactly \eqref{t'.t}.
		
	\end{enumerate}
	
\end{proof}

The first property of the previous Lemma indicates that series controlled by $\Phi$ are analytic. Conversely the following Lemma proves that analytic functions are controlled by appropriate series:

\begin{lemma}
	\label{lemma.H}
	Let $H(t,x,u)$ an analytic function in the neighborhood of $(0,0,0)\in\R\times\R^{d}\times\R^{N}$. Then there are some positive constants $C_{H}$, $R_H$, $\rho_H$ and $a_H$ such that
	\be
		\label{prec.H}
		H(t,x,u) \prec C_H \Phi(R_H X + \rho_H t) \prod_{j=1}^{N}\frac{\dsp 1}{1 - a_H u_j}
	\ee
\end{lemma}

\begin{proof}

	Formally we write
	\begin{eqnarray*}
		H(t,x,u) & = & \sum_{k_1,k_2,k_3} H_{k_1,k_2,k_3} \,t^{k_1} \,x^{k_2} \, u^{k_3} 
	\end{eqnarray*}
	
	\noindent with $k_1\in\N$, $k_2\in\N^{d}$ and $k_3\in\N^{N}$. By the Cauchy relations for $H$, we know there are some positive constants $C$, $r_1$, $r_2$ and $r_3$ depending only on $H$ such that
	$$
		|H_{k_1,k_2,k_3}| \leq  C\frac{\dsp 1}{\dsp r_1^{k_1}r_2^{|k_2|}r_3^{|k_3|}} \quad , \quad \forall \,( k_1, k_2,k_3)\in\N \times\N^{d} \times \N^{N}.
	$$
		
\noindent	We compare $ |H_{k_1,k_2,k_3}| $ to the coefficients of the series $\Phi(R_HX + \rho_H t)\prod \left(1 - a_H u_j\right)^{-1}$:
	\begin{align*}
		& \Phi(R_HX + \rho_H t) \prod_{j=1}^{N}\frac{\dsp 1}{1 - a_Hu_j} \\
		& = \sum_{p\in\N} \frac{c_0}{p^2+1} \left(R_HX+\rho_Ht\right)^{p} \, \sum_{q\in\N^{N}} a_H^{|q|}u^{q} \\
		& = \sum_{p} \sum_{k_1 + |k_2| = p} \frac{c_0}{p^2+1} \binom{p}{k_1,k_2} (\rho_H t)^{k_1} (R_HX)^{k_2} \, \sum_{q} a_H^{|q|}u^{q} \\
		& = \sum_{k_1,k_2,k_3} \frac{c_0}{(k_1 + |k_2|)^2+1} \binom{k_1 + |k_2|}{k_1,k_2} \rho_H^{k_1} R_H^{|k_2|} a_H^{|k_3|} \, t^{k_1} X^{k_2}  u^{k_3} 
	\end{align*}
	
	\noindent Then we have for all $( k_1, k_2,k_3)\in\N \times\N^{d} \times \N^{N}$ we have
	\begin{eqnarray*}
		|H_{k_1,k_2,k_3}| & \leq & C\frac{\dsp 1}{\dsp r_1^{k_1}r_2^{|k_2|}r_3^{|k_3|}} \\
		  & \leq & \frac{C}{c_0} \frac{\dsp (k_1+|k_2|)^2 + 1}{\dsp (\rho_Hr_1)^{k_1} (R_Hr_2)^{|k_2|} (a_Hr_3)^{|k_3|}} \, \frac{c_0}{(k_1+|k_2|)^2+1} \binom{k_1+|k_2|}{k_1,k_2}  \rho_H^{k_1} R_H^{|k_2|} a_H^{|k_3|}
	\end{eqnarray*}
	
	\noindent thanks to $\binom{k_1+|k_2|}{k_1,k_2} \geq 1$ for all $k_1,\,k_2$. By choosing $R_H$, $\rho_H$ and $a_H$ such that $\rho_Hr_1$, $R_Hr_2$ and $a_Hr_3$ are larger than $1$, the term
	$$
		\frac{\dsp (k_1+|k_2|)^2 + 1}{\dsp (\rho_Hr_1)^{k_1} (R_Hr_2)^{|k_2|} (a_Hr_3)^{|k_3|}}
	$$
	
	\noindent is bounded for all $( k_1, k_2,k_3)\in\N \times\N^{d} \times \N^{N}$. Then there is a constant $C_H>0$ depending only on $H$, $R_H$, $\rho_H$ and $a_H$ such that for all $( k_1, k_2,k_3)\in\N \times\N^{d} \times \N^{N}$ there holds
	$$
		|H_{k_1,k_2,k_3}| \leq C_H \frac{c_0}{(k_1 + |k_2|)^2+1} \binom{k_1 + |k_2|}{k_1,k_2}  \rho_H^{k_1} R_H^{|k_2|} a_H^{|k_3|}
	$$
	
	\noindent which implies
	$$
		H(x,t,u) \prec C_H \Phi(R_H X + \rho_H t) \prod_{j=1}^{N}\frac{\dsp 1}{1 - a_H u_j}.
	$$
	
\end{proof}

\begin{lemma}
	\label{lemma.c_1}
	There is $c_1>0$ such that
	\be
		\label{definition.c_1}
		\sum_{p\in\mathbb{Z}} \frac{\displaystyle c_1}{\displaystyle p^2+1} \frac{\displaystyle c_1}{\displaystyle (n-p)^2+1} \leq \frac{\displaystyle c_1}{\displaystyle n^2+1}
	\ee
\end{lemma}

\begin{proof}
	In the same way of the proof of the third point of Lemma \ref{majoring.properties}, there holds
	$$
		\sum_{p\in\mathbb{Z}} \frac{\displaystyle n^2+1}{\displaystyle (p^2+1)((n-p)^2+1)} \leq \sum_{p\in\mathbb{Z}} \frac{\displaystyle 2(p^2+1 + (n-p)^2 +1)}{\displaystyle (p^2+1)((n-p)^2+1)} \leq 4 \sum_{p\in\Z} \frac{\dsp 1}{\dsp p^2+1}
	$$
	
	\noindent which suffices to end the proof.
\end{proof}


\subsection{Definitions of functional spaces}


\subsubsection{Fixed time spaces $\EE_{s}$}

We consider trigonometric series in one variable $\theta$ with coefficients in the space of formal series in $d$ variables $x$ in the sense of Section \ref{subsection.majoring}, and we denote $F_{d+1}$ the space of all such trigonometric series:
$$
	F_{d+1} = \left\{ \v(x,\theta) = \sum_{n\in\mathbb{Z}} \v_n(x) e^{in\theta} \;\Big|\; \v_n(x) = \sum_{k\in\N^{d}} \v_{n,k} x^{k} \right\} .
$$ 

\begin{defi}[Fixed time spaces $\EE_{s}$]
	\label{definition.EE_s}
	Given $s\in[0,(\e \rho)^{-1})$, $R>0$, $\rho >0$, $M'>0$ and $\beta \in (0,1)$, we denote $\EE_{s} = \EE_{s}(R,\rho,M',\beta)$ the space of trigonometric series $\v\in F_{d+1}$ such that for some constant $C>0$ there holds
	\be
		\label{defi.EEs}
		\v_n(x) \prec C\frac{\displaystyle c_1}{\displaystyle n^2+1}\exp\Big({}-\big(M' - \int_{0}^{s} \g(\tau)d\tau \big)\left< n \right>\Big) \Phi\left(R X+\e\rho s\right) \quad ,\quad \forall n \in\Z .
	\ee
	\noindent where we denote
	\be
		\label{def.g}
		\g(\tau) = \g(\tau \,;R,\omega) := \g^{\sharp}(\tau \, ; R,\omega) + \beta .
	\ee
	
	\noindent We define a norm on $\EE_{s}$ with
	\be
		\label{defi.norm.s}
		\|\v\|_{s} = \inf \left\{ C>0 \,|\, \quad \eqref{defi.EEs} \text{  is satisfied  } \right\} .
	\ee
\end{defi}

Note that in definition \eqref{def.g} of $\g$, the function $\g^{\sharp}$ corresponds to either one defined in Lemma \ref{lemma.growth.propa}. In previous Definition \ref{definition.EE_s}, it is implicit that space $\EE_{s}$ depends on a positive function $\g^{\sharp}$.

Thanks to Lemma \ref{lemma.phi.prec.Phi}, for $s\in[0,(\e\rho)^{-1})$, all $\v\in\EE_{s}$ are holomorphic in the $x$ variable in the domain $\Omega_{R,\e\rho,s}$ defined by \eqref{defi.Omega.t}. We introduce also the growth time $\und{s}_1$ defined implicitely as
\be
	\label{def.und.s.1}
	M' = \int_{0}^{\und{s}_1} \g(\tau) d\tau .
\ee

\noindent For $0 \leq s < \und{s}_1$ we have $M' - \int_{0}^{s} \g(\tau)d\tau > 0$ and then analyticity of $\v$ in the $\theta$ variable. We will also see in Lemma \ref{lemma.algebra} that if $0\leq s < \und{s}_1$, the space $(\EE_{s},||\cdot||_{s})$ is an algebra. After these considerations it is convenient to define the final time as
\be
	\label{finaltime}
	\und{s} = \min\left\{\und{s}_1, \left(\e\rho\right)^{-1}\right\} .
\ee

To simplify the notations, in all the following we will omit the parameters $R$, $\rho$, $M'$ and $\beta$ in $\EE_{s}(R,\rho,M',\beta)$. All properties of spaces $\EE_{s}$ do not depend on particular values of those parameters.

\subsubsection{Spaces $\EE$}

We consider now trigonometric series
$$
	\u(s,x,\theta) = \sum_{n\in\Z} \u_n(s,x) e^{in\theta}
$$

\noindent with coefficients $\u_n(s,x)$ being formal series in $x$ whose coefficients depend smoothly on $s\in[0,\und{s})$. We denote $F_{d+2}$ the space of all such trigonometric series:
$$
	F_{d+2} = \left\{ \u(s,x,\theta) = \sum_{n\in\Z} \u_n(s,x) e^{in\theta} \;\Big|\; \u_n(s,x) = \sum_{k\in\N^{d}} \u_{n,k}(s)x^{k} \; \text{with } \u_{n,k}(s) \;C^{\infty} \text{ in } s \right\} .
$$

\begin{defi}[Spaces $\EE$]
	We introduce
	\be
		\label{defi.space}
		\EE = \left\{\u\in F_{d+2} \,|\, \forall \, 0\leq s <\und{s} \, , \quad \u(s)\in\EE_{s} \right\}
	\ee
	\noindent and the corresponding norm
	\be
		\label{defi.norm}
		||| \u ||| = \sup_{0\leq s < \und{s}} \|\u(s)\|_s .
	\ee
\end{defi}

Recalling the definition of majoring series \eqref{def.majoring.series} and the definition of $\EE_{s}$ \eqref{defi.EEs}, for all $\u\in\EE$ there holds
\be
	\label{esti.u}
	\u_n(s,x) \prec_{s} |||\u||| \frac{\displaystyle c_1}{\displaystyle n^2+1}\exp\Big({}-\big(M' - \int_{0}^{s} \g(\tau) d\tau \big)\left< n \right>\Big) \Phi\left(R X+\e\rho s\right) \; ,  \quad \forall n \in\Z ,\; \forall s\in[0,\und{s}) .
\ee


For $\u$ valued in $\mathbb{C}^{N}$, $\u\in\EE$ means simply that each component of $\u$ is in $\EE$, and $||| \u |||$ is then the maximum of the norms of the components. 

We denote the ball of $\EE$ of radius $a$, centered in $\u\in\EE$ by
\be
	\label{defi.ball}
	B_{\EE}(\u,a) = \left\{ \v\in\EE \;|\; |||\v - \u||| < a \right\} .
\ee


\subsection{Some properties of spaces $\EE$}

\subsubsection{The spaces $\EE_{s}$ are Banach spaces}

\begin{prop}
	For all $s\in[0,\und{s})$, the space $\EE_{s}$ equipped with the norm $||\cdot||_{s}$ is a Banach space.
\end{prop}

\begin{proof}
	Any $\v$ in $\EE_{s}$ is uniquely determined by the sequence of coefficients $(\v_{n,k})_{n\in\Z,k\in\N^{d}}$, where
	$$
		\v(x,\theta) = \sum_{n\in\Z} \v_n(x) e^{in\theta} \quad \text{with} \quad \v_n(x) = \sum_{k\in\N^{d}} \v_{n,k}x^{k} .
	$$ 
	
	\noindent By the definition of majoring series \eqref{def.maj} and notation \eqref{def.Phi_k.t}, the definition \eqref{defi.EEs} is equivalent to
	$$
		|\v_{n,k}| \leq C \frac{\displaystyle c_1}{\displaystyle n^2+1}\exp\Big({}-\big(M' - \int_{0}^{s} \g(\tau) d\tau  \big)\left< n \right>\Big) \Phi_{k}(\e s) \quad ,\quad \forall n\in\Z ,\; k\in\N^{d} ,\; 0\leq s < (\e \rho)^{-1} 
	$$
	
	\noindent where $\g$ is defined in \eqref{def.g}. Thus the map
	\be
		\mathcal{O}(s): \v\in\EE_{s} \mapsto \left(\v_{n,k} \mathcal{O}_{n,k}(s)\right)_{n\in\Z, k\in\N^{d}} 
	\ee
	\noindent with
	$$
		\mathcal{O}_{n,k}(s) = \left(\frac{\displaystyle c_1}{\displaystyle n^2+1}\exp\Big({}-\big(M' - \int_{0}^{s} \g(\tau) d\tau \big)\left< n \right>\Big) \Phi_{k}(\e s)\right)^{-1}
	$$
	
	\noindent is onto $ \ell^{\infty}(\C^{\Z\times\N^{d}}) $. By definition of the norm in $\EE_{s}$, the map $\mathcal{O}(s)$ is clearly an isometric isomorphism between $\EE_{s}$ and $\ell^{\infty}(\C^{\Z\times\N^{d}})$. This implies that $\left(\EE_{s},||\cdot||_{s}\right)$ is a Banach space.
\end{proof}

This implies immediately the following

\begin{coro}	
	The space $(\EE,|||\cdot|||)$ is a Banach space. 
\end{coro}

\subsubsection{The spaces $\EE_s$ are Banach algebra}

\begin{lemma}
	\label{lemma.algebra}
	For all $s\in[0,\und{s})$, for all $\v$ and $\w$ in $\EE_{s}$, the product $ \v\w$ is in $\EE_{s} $ and we have
	\be
		||\v \w||_{s} \leq ||\v||_{s} \, ||\w||_{s} .
	\ee
\end{lemma}

\begin{proof}
	Starting with the definition of $\EE_{s}$ \eqref{defi.EEs}, we obtain first for all $n\in\Z$ the following
	\begin{eqnarray*}
		(\v \w)_{n}(x) & = & \sum_{p+q=n}  \v_p(x) \w_q(x) \\
			& \prec & \sum_{p+q=n} 
				{\begin{aligned}[t]
				& ||\v||_{s} \frac{\displaystyle c_1}{\displaystyle p^2+1}\exp\Big({}-(M' - \int_{0}^{s} \g(\tau) d\tau )\left< p \right>\Big) \Phi\left(R X+\e\rho s\right) \\
				\times & ||\w||_{s} \frac{\displaystyle c_1}{\displaystyle q^2+1}\exp\Big({}-(M' - \int_{0}^{s} \g(\tau) d\tau )\left< q \right>\Big) \Phi\left(R X+\e\rho s\right)
				\end{aligned}} \\
			& \prec & ||\v||_{s} \, ||\w||_{s}\,\Phi^2\left(R X+\e\rho s\right) \,\sum_{p+q=n} \frac{\displaystyle c_1}{\displaystyle p^2+1}\frac{\displaystyle c_1}{\displaystyle q^2+1}\exp\Big({}-(M' - \int_{0}^{s} \g(\tau) d\tau )(\left< p \right>+\left< q \right>)\Big) .
	\end{eqnarray*}
	
	\noindent Recalling that $\Phi^2 \prec \Phi$ by Lemma \ref{majoring.properties}, we have 
	\begin{eqnarray*}
		(\v \w)_{n}(x) & \prec & ||\v||_{s} \, ||\w||_{s}\,\Phi\left(R X+\e\rho s\right)\, \sum_{p+q=n} \frac{\displaystyle c_1}{\displaystyle p^2+1}\frac{\displaystyle c_1}{\displaystyle q^2+1}\exp\Big({}-(M' - \int_{0}^{s} \g(\tau) d\tau )(\left< p \right>+\left< q \right>)\Big) \\
		  & \prec & ||\v||_{s} \, ||\w||_{s} \Phi\left(R X+\e\rho s\right) \exp\Big({}-(M' - \int_{0}^{s} \g(\tau) d\tau )\left< n \right>\Big)  \sum_{p+q=n} \frac{\displaystyle c_1}{\displaystyle p^2+1}\frac{\displaystyle c_1}{\displaystyle q^2+1} 
	\end{eqnarray*}
	
	\noindent because $\left< p \right>+\left< q \right> \geq \left< p+q\right> = \left<n\right>$ and $ M' - \int_{0}^{s} \g(\tau) d\tau  $ is positive for all $s < \und{s}$, and $\g$ is defined in \eqref{def.g}. And by definition \eqref{definition.c_1} of $c_1$ we have finally
	\begin{eqnarray*}
		(\v \w)_{n}(x) & \prec & ||\v||_{s} \, ||\w||_{s}\frac{\displaystyle c_1}{\displaystyle n^2+1}\exp\Big({}-(M' - \int_{0}^{s} \g(\tau) d\tau )\left< n \right>\Big) \Phi\left(R X+\e\rho s\right)
	\end{eqnarray*}
	
	\noindent which implies the result.

\end{proof}

This implies immediately the following

\begin{coro}
	\label{coro.algebra}
	The space $\EE$ is an algebra, and the norm $|||\cdot|||$ is an algebra norm.
\end{coro}

\subsubsection{Action of holomorphic functions}


\begin{lemma}
	\label{lemma.holomorphic}
	Let $H(t,x,u)$ be a holomorphic function on a neighborhood of $(0,0,0)\in\R_{t}\times\R_{x}^{d}\times\R_{u}^{N}$. Then for $\e$ small enough there are constants $C_H$, $R_H$ and $\rho_H$ which depend only on $H$ and $c_0$, such that for all $R\geq R_H$ and $\rho\geq\rho_H$,  
	\be
		\label{esti.holomorphic}
		\forall \,\u\in B_{\EE(R,\rho)}(0,1): \quad |||\H(\u)|||\leq C_{H}2^{N}
	\ee
	
	\noindent where $\H$ is defined by \eqref{def.mathbf.H} and $|||\cdot|||$ is defined by \eqref{defi.norm}.

\end{lemma}

\begin{proof}
	
	Thanks to Lemma \ref{lemma.H} we have
	$$
		H(t,x,u) \prec C_H \Phi(R_H X + \rho_H t) \prod_{j=1}^{N}\frac{\dsp 1}{1 - a_H u_j}
	$$
	
	\noindent Let $\u$ be in $ B_{\EE}(0,1)$ with $\EE=\EE(R,\rho)$ for $R \geq R_H$ and $\rho \geq \rho_H$. For $\e$ small enough we have $ \e a_{H} < 1/2 $ so that $ ||| \e a_H \u ||| \leq 1/2 $. We now prove that $\H(s,x,\u)$ is indeed in $\EE$. By Lemma \ref{lemma.H} it suffices to prove that
	$$ (s,x,\theta) \mapsto C_H \Phi(R_H X + \e\rho_H s) \prod_{j=1}^{N}\frac{\dsp 1}{1 - \e a_H \u_j(s,x,\theta)} $$
	
	\noindent is in $\EE$. Because $\EE$ is a Banach algebra (Corollary \ref{coro.algebra})and $\e a_{H} < 1/2$, the operator 
	$$ \u \mapsto \prod_{j=1}^{N}\left(1 - \e a_H \u_j\right)^{-1} $$
	
	\noindent is a bounded operator and we have
	$$
		\left|\left|\left|\prod_{j=1}^{N}\frac{\dsp 1}{1 - \e a_H \u_j(s,x,\theta)}\right|\right|\right| \leq \prod_{j=1}^{N}\frac{\dsp 1}{1 - \e a_H |||\u|||} \leq \left(\frac{\dsp 1}{1 - 1/2}\right)^{N} = 2^{N}
	$$
	
	\noindent By \eqref{R_0.R}, we have $\Phi(R_H X + \e\rho_H s) \prec_{s} \Phi(R X+\e\rho s)$ for all $R\geq R_H$ and $\rho\geq\rho_H$, so that 
	\begin{eqnarray*}
		\Phi(R_H X + \e\rho_H s)  \Phi(R X+\e\rho s) & \prec_{s} & \Phi(R X+\e\rho s)^2 \\
			& \prec_{s} & \Phi(R X+\e\rho s)
	\end{eqnarray*}
	\noindent by \eqref{Phi2}. Hence $(s,x,\theta) \mapsto C_H \Phi(R_H X + \e\rho_H s) \prod_{j=1}^{N}(1 - \e a_H \u_j(s,x,\theta))^{-1}$ is in $\EE$, and then for all $\u\in\EE$ in the ball $B_{\EE}(0,1)$ the bound \eqref{esti.holomorphic} holds.
	
\end{proof}

%

In the operators $T^{[\theta]}$, $T^{[x]} $ and $T^{[\u]} $ defined by \eqref{def.T_theta}, \eqref{def.T_x} and \eqref{def.T_u}, there appear $A$, $\und{A}$, $A_{j}$ and $F$. In Corollary \ref{reg_dtheta.bis}, there will appear also $A_{u_j}$, all of which are analytic functions in variables $(t,x,u)\in\R\times\R^{d}\times\R^{N}$ in a neighborhood of $(0,0,0)  \in \R_{t}\times\R_{x}^{d}\times\R_{u}^{N}$. The previous Lemma applies:

\begin{coro}
	\label{coro.norms}
	There are constants $R_0$ and $\rho_0$ such that for all $R\geq R_0$, $\rho \geq \rho_0$ and $\e$ small enough:
	\be
		\label{action.A}
		\forall \,\u\in B_{\EE(R,\rho)}(0,1): \qquad |||\H(\u) ||| \lesssim 1
	\ee
	\noindent with $H$ equals to $A$, $\und{A}$, $A_j$, $F$, or $A_{u_j}$.
\end{coro}



\subsection{Action of $U(s',s)$ on $\EE$}


Recall the growth of the Fourier modes of the propagator as showed in Lemma \ref{lemma.growth.propa}
$$
	|U_n(s',s,x)| \lesssim \omega^{-(m-1)}\,\exp\left( |n| \int_{s'}^{s} \g^{\sharp}(\tau) d\tau  \right) .
$$

\noindent Here, as opposed to \cite{metivier2005remarks}, the propagator $U_n$ does depend on $x$. As $U_n(s',s,x)$ is the solution of the differential equation \eqref{equation.propagator} and as $\und{A}(t,x)$ is analytic in $x$, so is $U_n(s',s,x)$. Using the Cauchy inequalities as in the proof of Lemma \ref{lemma.H}, we can prove in particular that 
\be
	\label{esti.prec.propa}
	U_n(s',s,x) \prec_{s',s} \omega^{-(m-1)}\,\exp\left( |n| \int_{s'}^{s} \g^{\sharp}(\tau) d\tau \right) \Phi(R_0 X)
\ee

\noindent for $R_0$ determined in Corollary \ref{coro.norms}. We use this result to determine precisely the action of the propagator on $\EE$.

\begin{lemma}
	\label{lemma.action.U}
	Given $\u$ in $\EE = \EE(R,\rho,M',\beta)$ then for all $n\in\Z$ ans $0 \leq s' \leq s < \und{s}$ there holds
	\be
		\label{precise.bound.action.U}
		U_n(s',s)\u_n(s',x) \prec_{s',s} C_n(s',s) \,\omega^{-(m-1)}\,||\u(s')||_{s'} \frac{\displaystyle c_1}{\displaystyle n^2+1} \, e^{{}-(M'-\int_{0}^{s}\g(\tau) d\tau)\left< n \right>} \Phi\left(R X+\e\rho s\right)
	\ee
	\be 
		\label{defi.Cn}
		\text{with} \quad C_n(s',s) = \exp\left({}- \left<n\right> \beta \,(s-s') \right) \leq 1 .
	\ee
	
	\noindent In particular we have
	\be
	\label{bound.lemma.action.U}
		\|U(s',s)\u(s')\|_s \leq \omega^{-(m-1)}\, \|\u(s')\|_{s'} \quad , \quad \forall \, 0 \leq s' \leq s < \und{s} .
	\ee
	
\end{lemma}

\begin{proof}
	By the estimate \eqref{esti.u} for $\u\in\EE$ we have
	$$
		\u_n(s',x) \prec_{s'} ||\u(s')||_{s'} \frac{\displaystyle c_1}{\displaystyle n^2+1}\exp\Big({}-(M' - \int_{0}^{s'} \g(\tau)d\tau )\left< n \right>\Big) \Phi\left(R X+\e\rho s'\right)
	$$
	
	\noindent where $\g$ is defined in \eqref{def.g}. By estimate \eqref{esti.prec.propa} and the multiplicative property of $\prec$ there holds
	\begin{eqnarray*}
		U_n(s',s)\u_n(s',x) & \prec_{s',s} & \omega^{-(m-1)}\,\exp\left(|n| \int_{s'}^{s} \g^{\sharp}(\tau) d\tau \right) \\
		& & \times||\u(s')||_{s'} \frac{\displaystyle c_1}{\displaystyle n^2+1}\exp\Big({}-(M'- \int_{0}^{s'} \g(\tau)d\tau )\left< n \right>\Big) \Phi\left(R X+\e\rho s'\right) \\
		  & \prec_{s',s} & \omega^{-(m-1)}\,||\u(s')||_{s'} \frac{\displaystyle c_1}{\displaystyle n^2+1}\exp\Big({}-(M' - \int_{0}^{s} \g(\tau)d\tau )\left< n \right>\Big) \Phi\left(R X+\e\rho s\right) \\
		  & & \times \exp\left({}- \left<n\right> \int_{s'}^{s} \left(\g(\tau) - \g^{\sharp}(\tau) \right) d\tau \right)
	\end{eqnarray*}
	
	\noindent because $\Phi(R X + \e\rho s') \prec_{s',s} \Phi(R X + \e\rho s)$ for $s' \leq s < \und{s}$ by \eqref{t'.t}. This gives us exactly \eqref{precise.bound.action.U} using \eqref{def.g}, and then \eqref{bound.lemma.action.U}.
 
\end{proof}

\begin{remark}
	The estimate \eqref{bound.lemma.action.U} is not precise enough to show that $T$ is a contraction in $\EE$. The more precise estimate \eqref{precise.bound.action.U} is very important for the estimate \eqref{esti.reg_dtheta} below.
\end{remark}


\subsection{Norm of the free solution}


\begin{lemma}[Norm of the free solution]
	\label{lemma.norm.free.solution}
	The free solution $\f$ defined by \eqref{def.initial.data} satisfies
	\be
		\label{norm.free.solution}
		|||\f||| \lesssim \omega^{-(m-1)}\,e^{M'-M(\e)} .
	\ee
\end{lemma}

\begin{proof}
	The Fourier decomposition of $\f_{\e}$ is given by $\f_{\e} = \f_{+1}e^{-i\theta} + \f_{-1}e^{i\theta}$ with $ \f_{\pm}(s,x) = U_{\mp}(0,s,x)\vec{e}_{\pm} $. The Fourier coefficients $\f_{\pm}$ satisfy thanks to \eqref{esti.prec.propa} the estimate
	\be
		\label{estimate.f.pm1} 
		\f_{\pm 1}(s) \prec_{s} \omega^{-(m-1)}\,e^{-M(\e)}e^{ \int_{0}^{s} \g^{\sharp}(\tau) d\tau } \Phi(R_0 X) .
	\ee
	
	\noindent Then by definition of $|||\cdot|||$ given by \eqref{defi.norm}, and by definition \eqref{def.g} of $\g$, there holds 
	\begin{eqnarray*}
		|||\f_{\pm 1}||| & = & \frac{\dsp 2}{\dsp c_0c_1} \omega^{-(m-1)}\,e^{M'-M(\e)} \max_{[0,\und{s})} e^{ \int_{0}^{s} \g^{\sharp}(\tau) d\tau}e^{- \int_{0}^{s} \g(\tau) d\tau} \\
			& = & \frac{\dsp 2}{\dsp c_0c_1} \omega^{-(m-1)}\,e^{M'-M(\e)} \max_{[0,\und{s})} e^{- \int_{0}^{s} \beta d\tau} \\
			& \lesssim & \omega^{-(m-1)}\,e^{M'-M(\e)}
	\end{eqnarray*}
	
	\noindent which ends the proof.
	
%
%
%
%
\end{proof}


\section{Regularization by integration in time and contraction estimates}
\label{section.contraction}


In this section we prove estimates in spaces $\EE$ for the three operators $T^{[\theta]}$, $T^{[x]}$ and $T^{[\u]}$ defined respectively by \eqref{def.T_theta}, \eqref{def.T_x} and \eqref{def.T_u}. Note that in the first two operators there appear derivation operators $\d_{\theta}$ and $\d_{x_j}$. As we will see in the next subsection, these are not bounded operators in $\EE$. But thanks to some smoothing effect of the time-integration, as used in \cite{metivier2005remarks}, we will show that operators $T^{[\theta]}$, $T^{[x]}$ and $T^{[\u]}$ are in fact bounded in $\EE$. We will follow in this section the work of \cite{ukai2001boltzmann}.


\subsection{Lack of boundedness of derivation operators}


In the following we make precise how the derivation operators $\d_{x_j}$ and $\d_{\theta}$ act on $\EE$.

\begin{lemma}[Estimates for the derivation operators]
	\label{lemma.derivation}
	For any $\u$ in $\EE$, we have the following estimates
	\begin{eqnarray}
		(\d_{\theta}\u)_n(s,x) & \prec_{s} & |n| \,|||\u||| \frac{\dsp c_1}{\dsp n^2+1} e^{{}-(M'- \int_{0}^{s}\g(\tau) d\tau) \left<n\right>} \Phi\left(R X+\e\rho s\right) \label{action.dtheta} \\
		(\d_{x_j}\u)_n(s,x) & \prec_{s} & \,R \,|||\u||| \frac{\dsp c_1}{\dsp n^2+1} e^{{}-(M'- \int_{0}^{s}\g(\tau) d\tau ) \left<n\right>} \Phi'\left(R X+\e\rho s\right) \label{action.dy}
	\end{eqnarray}
	
	\noindent for all $n\in\Z$ and $s\in[0,\und{s})$. 
	
\end{lemma}

\begin{proof}
	The estimates \eqref{action.dtheta} and \eqref{action.dy} are straightforward. Indeed $	(\d_{\theta}\u)_n = n \u_n $ for all $n\in\Z $ which implies \eqref{action.dtheta}. For \eqref{action.dy} there holds $	(\d_{x_j}\u)_n = \d_{x_j} \u_n $ for all $ n\in\Z $ and we get \eqref{action.dy} thanks to the relation \eqref{dx.Phi}.
	
\end{proof}

\begin{remark}[Lack of boundedness of derivation operators]
	Lemma {\rm \ref{lemma.derivation}} does not prove directly that the $\d_{x_j}$ and $\d_{\theta}$ are not bounded operators on $\EE$. But let us consider the function in $\EE$ defined by its Fourier modes
	$$ \u_n(s,x) = \frac{\dsp c_1}{\dsp n^2+1} e^{{}-(M'- \int_{0}^{s}\g(\tau) d\tau ) \left<n\right>} \Phi\left(R X+\e\rho s\right) \quad \forall n\in\Z $$
	\noindent Then
	$$ 
		\left(\d_{\theta}\u\right)_n(s,x) = \frac{\dsp c_1 n}{\dsp n^2+1} e^{{}-(M'- \int_{0}^{s}\g(\tau) d\tau ) \left<n\right>} \Phi\left(R X+\e\rho s\right) 
	$$
	
	\noindent and $\d_{\theta}\u$ is not in $\EE$ as we may not bound $ \frac{|n|}{n^2+1}$ by $ \frac{1}{n^2+1} $. Since $\Phi' \prec \Phi$ does not hold, the applications $\d_{x_j} \u$ are not in $\EE$ either. Hence the derivation operators $\d_{x_j}$ and $\d_{\theta}$ are not bounded operators in $\EE$.

\end{remark}

In the following, we will need exact estimates on terms like $\v\d_{\theta}\u$, or $U(s',s)\d_{x_j}\u(s')$.

\begin{lemma}[Action of product and $U(s',s)$ on the lack of boundedness]
	\label{lemma.product.lack}
	
	For any $\u$ and $\v$ in $\EE$, for all $n\in\Z$ and $0\leq s' \leq <\und{s}$, there holds
	\begin{align}
		& (\v\d_{\theta}\u)_n(s,x) \prec_{s} C|n| \,|||\u|||\, |||\v||| \frac{\dsp c_1}{\dsp n^2+1} e^{{}-(M'- \int_{0}^{s}\g(\tau) d\tau ) \left<n\right>} \Phi\left(R X+\e\rho s\right) \label{action.product.dtheta} \\
		& (\v\d_{x_j}\u)_n(s,x) \prec_{s} \,C'R\, |||\u|||\, |||\v||| \frac{\dsp c_1}{\dsp n^2+1} e^{{}-(M'- \int_{0}^{s}\g(\tau) d\tau ) \left<n\right>} \Phi'\left(R X+\e\rho s\right) \label{action.product.dy} \\
		& (U(s',s,x,\theta)\d_{x_j}\u(s',x,\theta))_n \prec_{s',s} C_n(s',s)R\, \omega^{-(m-1)}\,||\u(s')||_{s'} \frac{\dsp c_1}{\dsp n^2+1} e^{{}-(M'- \int_{0}^{s}\g(\tau) d\tau ) \left<n\right>} \Phi'\left(R X+\e\rho s'\right) \label{action.U.dy}
	\end{align}
	
	\noindent for some constants $C>0$ and $C'>0$ independent of all parameters.
\end{lemma}

\begin{proof}

	To prove estimate \eqref{action.product.dtheta} it suffices to get back to the proof of Lemma \ref{lemma.algebra}. Following the same computations we get
	$$
		(\v\d_{\theta}\u)_n(s,x) \prec_{s} ||\u||_{s} \, ||\v||_{s} \Phi\left(R X+\e\rho s\right) \exp\Big({}-(M' - \int_{0}^{s}\g(\tau) d\tau )\left< n \right>\Big)  \sum_{p+q=n} \frac{\displaystyle c_1}{\displaystyle p^2+1}\frac{\displaystyle c_1 |q|}{\displaystyle q^2+1} .
	$$ 
	
	\noindent By adaptating the proof of the existence of some $c_1$ such that \eqref{definition.c_1} in Lemma \ref{lemma.c_1} there holds
	$$ 
		\sum_{p+q=n} \frac{\dsp c_1}{\dsp p^2+1} \frac{\dsp c_1|q|}{\dsp q^2+1} \lesssim \frac{\dsp c_1 |n|}{\dsp n^2+1} \quad , \quad \forall n\in\mathbb{Z} 
	$$
	
	\noindent and then \eqref{action.product.dtheta} holds.

	In the same way we have
	$$
		(\v\d_{x_j}\u)_n(s,x) \prec_{s} ||\v||_{s} \, ||\w||_{s} \frac{\displaystyle c_1}{\displaystyle p^2+1}\exp\Big({}-(M' - \int_{0}^{s}\g(\tau) d\tau )\left< n \right>\Big)\,R\Phi'\left(R X+\e\rho s\right) \Phi\left(R X+\e\rho s\right) 
	$$
	
	\noindent Thanks to Lemma \ref{majoring.properties}, we differentiate the inequality $\Phi^2 \prec \Phi$ to get $ 2\Phi \Phi' \prec \Phi' $, hence estimate \eqref{action.product.dy}.
	
	For estimate \eqref{action.U.dy} it suffices to adapt the proof of Lemma \ref{lemma.action.U}, as $U(s',s)$ acts only on the size of the Fourier coefficients $\u_n(s,x)$ and not on the coefficients of the series $\u_{n,k}(s)$.

\end{proof}


\subsection{Integration in time and regularization of $\d_{\theta}$}


\begin{prop}
	\label{reg_dtheta}
	For operator $T^{[\theta]}$ defined by \eqref{def.T_theta}, for any $\u\in B_{\EE}(0,1)$ there holds
	\be
		\label{esti.reg_dtheta}
		|||T^{[\theta]}(\u)||| \lesssim  \omega^{-(m-1)}\,\beta^{-1} |||({\mathbf A} - \und{{\mathbf A}})(\u)|||\,|||\u||| .
	\ee
\end{prop}

\begin{proof}	

By Lemma \ref{lemma.holomorphic}, the function $(\A - \und{\A})(\cdot,\u)$ is in $\EE$. Applying first estimate \eqref{action.product.dtheta} we get
\begin{align*}
	& \left((\A - \und{\A})(s',\u(s')) \d_{\theta}\u(s')\right)_n \\
	& \prec_{s'} |n| \,|||\u|||\, |||({\mathbf A} - \und{{\mathbf A}})(\u)||| 
\frac{\dsp c_1}{\dsp n^2+1} e^{{}- (M' - \int_{0}^{s'}\g(\tau) d\tau) \left<n\right>} \Phi\left(R X+\e\rho s'\right) 
\end{align*}

\noindent where $\g$ is defined in \eqref{def.g}. Then by \eqref{action.U.dy} there holds
\begin{align*}
	& \left(U(s',s) (\A - \und{\A})(s',\u(s')) \d_{\theta}\u(s') \right)_n \\
	& \prec_{s',s} C_n(s',s) |n| \,\omega^{-(m-1)}\,|||\u|||\, |||({\mathbf A} - \und{{\mathbf A}})(\u)||| \frac{\dsp c_1}{\dsp n^2+1}e^{{}- (M' - \int_{0}^{s}\g(\tau) d\tau) \left<n\right>} \Phi\left(R X+\e\rho s\right) .
\end{align*}

\noindent As integration in time and Fourier transform commute, we have
$$
	\left(T^{[\theta]}(\u)\right)_n(s) = \int_{0}^{s} \left(U(s',s) (\A - \und{\A}) \d_{\theta}\u(s') \right)_n ds' 
$$

\noindent and then
\begin{align*}
	& \left(T^{[\theta]}(\u)\right)_n(s) \\
	& \prec_{s} \int_{0}^{s} C_n(s',s) |n| \,\omega^{-(m-1)}\,|||\u||| \, |||({\mathbf A} - \und{{\mathbf A}})(\u)||| \,\frac{\dsp c_1}{\dsp n^2+1} e^{{}- (M'- \int_{0}^{s}\g(\tau) d\tau ) \left<n\right>} \Phi\left(R X+\e\rho s\right) ds' \\
	& \prec_{s} \omega^{-(m-1)}\,|||\u|||\,  |||({\mathbf A} - \und{{\mathbf A}})(\u)||| \,\frac{\dsp c_1}{\dsp n^2+1}e^{{}- (M'- \int_{0}^{s}\g(\tau) d\tau ) \left<n\right>} \Phi\left(R X+\e\rho s\right) \int_{0}^{s} C_n(s',s) |n| \, ds' .
\end{align*}

\noindent To end the proof, we prove a uniform bound independent of $n$ for the integral term $ \int_{0}^{s} C_n(s',s) |n| \, ds'$. Recalling first the definition \eqref{defi.Cn}:
$$
	C_n(s',s) = \exp\left( - \beta\, (s-s') \left<n\right> \right)
$$

\noindent there holds
\begin{eqnarray*}
	\int_{0}^{s} C_n(s',s) |n| \, ds' & = & \int_{0}^{s} \exp\left( - \beta\, (s-s') \left<n\right> \right) |n| ds' \\
		& = & \exp\left( {} - \beta\, s\left<n\right> \right) \int_{0}^{s} \exp\left( \beta\, s' \left<n\right> \right) |n| ds' \\
		& \leq & \beta^{-1}
\end{eqnarray*}

\noindent which ends the proof.
	
\end{proof}

Thanks to the definition \eqref{def.undA} of $\und{A}$ and an expansion formula we make the previous result more precise:

\begin{coro}
	\label{reg_dtheta.bis}
	For operator $T^{[\theta]}$ defined by \eqref{def.T_theta}, for any $\u\in B_{\EE}(0,1)$ there holds
	\be
		\label{esti.reg_dtheta.bis}
		|||T^{[\theta]}(\u)||| \lesssim  \omega^{-(m-1)}\,\beta^{-1}\, \e \,|||\u|||^2 .
	\ee
\end{coro}

\begin{proof}
	By analyticity of $A(t,x,u)$ there are a family of matrices $A_{u_j}(t,x,u)$ depending analytically on $(t,x,u)$ such that
	$$
		A(t,x,u) - \und{A}(t,x) = \sum_j A_{u_j} u_j.
	$$
	
	\noindent This implies that
	$$
		|||({\mathbf A} - \und{{\mathbf A}})(\u)||| \leq \e |||\u|||
	$$
	
	\noindent by definition of notation \eqref{def.mathbf.H}.
\end{proof}


\subsection{Integration in time and regularization of $\d_{x_j}$}


After managing to deal with unbounded term $\d_{\theta}\u$ we consider the other unbounded terms $\d_{x_j}\u$. We consider then the operator $T^{[x]}$:

\begin{prop}
	\label{reg_dx}
	For operator $T^{[x]}$ defined by \eqref{def.T_x} and any $\u\in B_{\EE}(0,1)$, there holds
	\be
		\label{esti.reg_dx}
		|||T^{[x]}(\u)||| \lesssim \omega^{-(m-1)}\,R\rho^{-1}\,|||\u||| .
	\ee
\end{prop}

\begin{proof}

By Lemma \ref{lemma.holomorphic}, functions ${\mathbf A}_j(\cdot,\cdot, \u(\cdot))$ are in $\EE$. Applying first estimate \eqref{action.product.dy} we get
\begin{align*}
	& \left(\A_j(s',\u(s')) \d_{x_j}\u(s')\right)_n \\
	& \prec_{s'} R \,|||\u|||\, |||\A_j(\u)||| 
\frac{\dsp c_1}{\dsp n^2+1} e^{{}- (M'- \int_{0}^{s'}\g(\tau)d\tau ) \left<n\right>} \Phi'\left(R X+\e\rho s'\right) 
\end{align*}

\noindent where we denote $|||\A_j(\u)|||$ for $|||\A_j(\cdot,\cdot, \u(\cdot))|||$. Then by Lemma \ref{lemma.action.U} there holds
\begin{align*}
	& \left(\sum_j U(s',s) \A_j(s',\u(s')) \d_{x_j}\u(s') \right)_n \\
	& \prec_{s',s} C_n(s',s) R \,\omega^{-(m-1)}\,|||\u|||\, \sum_j|||\A_j(\u)||| \frac{\dsp c_1}{\dsp n^2+1}e^{{}- (M'- \int_{0}^{s}\g(\tau)d\tau ) \left<n\right>} \Phi'\left(R X+\e\rho s'\right) \\
	& \prec_{s',s} R \,\omega^{-(m-1)}\,|||\u|||\, \sum_j|||\A_j(\u)||| \frac{\dsp c_1}{\dsp n^2+1}e^{{}- (M'- \int_{0}^{s}\g(\tau)d\tau ) \left<n\right>} \Phi'\left(R X+\e\rho s'\right)
\end{align*}

\noindent as $C_n(s',s) \leq 1$. As integration in time and Fourier transform commute, we have
$$
	\left(T^{[x]}(\u)\right)_n(s) = \int_{0}^{s} \left(U(s',s) \e \sum_j\A_j(s', \u(s')) \d_{x_j}\u(s')\right)_n ds' 
$$

\noindent and then
\begin{align*}
	& \left(T^{[x]}(\u)\right)_n \\
	& \prec_{s} \int_{0}^{s} \e R \,\omega^{-(m-1)}\,|||\u|||\, \sum_j|||\A_j(\u)||| \frac{\dsp c_1}{\dsp n^2+1}e^{{}- (M'- \int_{0}^{s}\g(\tau)d\tau ) \left<n\right>} \Phi'\left(R X+\e\rho s'\right) ds' \\
	& \prec_{s} \e R\omega^{-(m-1)}\,|||\u|||\, \sum_j|||\A_j(\u)||| \frac{\dsp c_1}{\dsp n^2+1}e^{{}- (M'- \int_{0}^{s}\g(\tau)d\tau ) \left<n\right>} \int_{0}^{s} \Phi'\left(R X+\e\rho s'\right) \, ds' .
\end{align*}

\noindent By term-wise integration of the series, we have
\begin{eqnarray*}
	\int_{0}^{s} \Phi'\left(RX+\e\rho s'\right) ds' & = & \int_{0}^{s} (\e\rho)^{-1} \d_{s'} \left( \Phi\left(R X+\e\rho s'\right) \right)ds' \\
	  & \prec_{s} & (\e\rho)^{-1} \Phi\left(R X+\e\rho s\right)
\end{eqnarray*}
\noindent which suffices to end the proof.

\end{proof}


\subsection{Integration in time and product}


As $\EE$ is an algebra the operator $T^{[\u]}$ is directly bounded, with no need of a regularization by time result, on the contrary of operators $T^{[\theta]}$ and $T^{[x]}$. The following proposition gives us precisely

\begin{prop}
	\label{reg_u}
	For the operator $T^{[u]}$ defined by \eqref{def.T_u}, for any $\u\in B_{\EE}(0,1)$ there holds
	\be
		\label{esti.reg_u}
		|||T^{[\u]}(\u)||| \lesssim \omega^{-(m-1)}\,\beta^{-1}\, \e \, |||\F(\u)||| \, |||\u||| .
	\ee
\end{prop}

\begin{proof}

As in the proof of Proposition \ref{reg_dtheta} we have
\begin{align*}
	& \left(T^{[\u]}(\u)\right)_n(s) \\
	& \prec_{s} \int_{0}^{s} C_n^{\eta}(s',s) \,\omega^{-(m-1)}\,|||\u|||\, \e|||\F(\u)||| \frac{\dsp c_1}{\dsp n^2+1}e^{{}- (M'- \int_{0}^{s} \g(\tau)d\tau ) \left<n\right>} \Phi\left(R X+\e\rho s\right) ds' \\
	& \prec_{s} \e \omega^{-(m-1)}\,|||\u|||\, |||\F(\u)||| \frac{\dsp c_1}{\dsp n^2+1}e^{{}- (M'- \int_{0}^{s} \g(\tau)d\tau ) \left<n\right>} \Phi\left(R X+\e\rho s\right) \int_{0}^{s} C_n(s',s) |n| \, ds'
\end{align*}

\noindent and as 
$$
	\int_{0}^{s} C_n(s',s) |n| \, ds' \lesssim \beta^{-1} \quad , \quad \forall n\in\Z ,\; \forall 0 \leq s <\und{s}
$$

\noindent we get \eqref{esti.reg_u}.

\end{proof}

Using Assumption \ref{hypo.4}, we have in fact a more precise estimate:

\begin{coro}
	\label{coro.ref.du}
	Under Assumption {\rm \ref{hypo.4}}, operator $T^{[\u]}$ defined by \eqref{def.T_u} satisfied for any $\u\in B_{\EE}(0,1)$ the following bound
	\be
		\label{esti.reg_u.bis}
		|||T^{[\u]}(\u)||| \lesssim \omega^{-(m-1)}\, \beta^{-1} \e \, |||\u|||^2 .
	\ee
	
\end{coro}


\subsection{Contraction estimates}


The three previous subsections give us some precious estimates on operators $T^{[\theta]}$, $T^{[x]}$ and $T^{[\u]}$ in $\EE$. In the perspective of using a fixed point theorem on the Banach space $\EE$, we prove now estimates on the differences $T^{[\theta]}(\u) - T^{[\theta]}(\v)$, $T^{[x]}(\u) - T^{[x]}(\v)$ and $T^{[\u]}(\u) - T^{[\u]}(\v)$ for $\u$ and $\v$ in the ball $B_{\EE}(0,1)$.

\begin{prop}[Contraction estimates in $\EE$]
	\label{prop.estimates}
	There are $R_0$, $\rho_0>0$ such that for all $R\geq R_0$, $\rho>\rho_0$ and $\e \in (0,1)$, we get the following estimates for all $\u$ and $\v$ in $B_{\EE}(0,1)$:
	\begin{align}
		\label{esti.T}
		& |||T(\u)||| \lesssim \omega^{-(m-1)}\,\left( \beta^{-1} \left( \e|||\F(\u)||| + |||\mathbf{A}(\u) - \und{\A}(\u)||| \right)  + R\rho^{-1}\right) |||\u||| \\
		\label{esti.TT}
		& |||T(\u) - T(\v)||| \lesssim \omega^{-(m-1)}\,\left( \beta^{-1} \left( \e |||\F(\u)||| + |||\mathbf{A}(\u) - \und{\A}(\u)||| \right)  + R\rho^{-1}\right)  |||\u - \v|||
	\end{align}
\end{prop}

\begin{proof}
	Recalling that $T = T^{[\theta]} + T^{[x]} + T^{[\u]}$, we can apply directly Propositions \ref{reg_dtheta}, \ref{reg_dx} and \ref{reg_u} to get \eqref{esti.T}. 
	
	To prove the contraction estimate \eqref{esti.TT}, we write for all $\u$ and $\v$ in $B_{\EE}(0,1)$ the following
$$
	T(\u) - T(\v) = \left(T^{[\theta]}(\u)-T^{[\theta]}(\v)\right) + \left(T^{[x]}(\u)-T^{[x]}(\v)\right) + \left(T^{[\u]}(\u)- T^{[\u]}(\v)\right)
$$

\noindent To get estimates on those three terms we first introduce some notations:
\begin{eqnarray*}
	T_{H}^{[\theta]}(s,\u) & = & \int_{0}^{s} U(s',s) \,\H(\u(s')) \,\d_{\theta} \u(s') ds' \\
	T_{H}^{[x_j]}(s,\u) & = & \int_{0}^{s} U(s',s)\, \H(\u(s')) \, \d_{x_j} \u(s') ds' \\
	T_{H}^{[\u]}(s,\u) & = & \int_{0}^{s} U(s',s) \, \H(\u(s')) \, \u(s') ds' 
\end{eqnarray*}

\noindent with $H(t,x,u)$ holomorphic on the neighborhood of $(0,0,0)\in\R_{t}\times\R_{x}^{d}\times\R_{u}^{N}$, and using notation \eqref{def.mathbf.H}. For example, 
\be
	\label{T_H_theta}
	T^{[\theta]}(s,\u) = T_{H}^{[\theta]}(s,\u) \qquad \text{with } H = A - \und{A}
\ee

\noindent Differences like $T^{[\theta]}(s,\u) - T^{[\theta]}(s,\v)$ are now easier to write. For example
\begin{eqnarray}
	T_{H}^{[\theta]}(s,\u) - T_{H}^{[\theta]}(s,\v) & = & \int_{0}^{s} U(s',s) \,\left(\H(\u(s')) \,\d_{\theta} \u(s') - \H(\v(s')) \,\d_{\theta} \v(s') \right) ds' \nonumber \\
		& = & \phantom{+}\int_{0}^{s} U(s',s) \,\left(\H(\u(s')) - \H(\v(s'))\right) \,\d_{\theta} \v(s') ds' \label{one} \\
		& & + \int_{0}^{s} U(s',s) \,\H(\u(s')) \,\d_{\theta} (\u-\v)(s')  ds'
\end{eqnarray}

\noindent and these two terms are very similar to $T_{H}^{[\theta]}$. The same proof as Proposition \ref{reg_dtheta} gives then directly 
$$
	\left|\left|\left| \int_{0}^{s} U(s',s) \,\H(\u(s')) \,\d_{\theta} (\u-\v)(s')  ds' \right|\right|\right| \lesssim  \beta^{-1} \,|||\H(\u)|||\, |||\u - \v|||
$$

\noindent For the other term \eqref{one} we first note that for all $(t,x,u)$ and $(t,x,v)$ close to the distinguished point $(0,0,0)\in\R\times\R^{d}\times\R^{N}$, with $u-v$ small enough, there holds
$$
	H(t,x,u) - H(t,x,v) = (u-v)\,\widetilde{H}(t,x,u,v)
$$

\noindent with
$$
	\widetilde{H}(t,x,u,v) = \int_{0}^{1} \d_{u} H(t,x,v + y (u-v)) dy .
$$

\noindent Note that $\widetilde{H}$ is an analytic function of $(t,x,u,v)$ near $(0,0,0,0)$. Hence an adaptation of the proof of Proposition \ref{reg_dtheta} gives
\begin{align*}
	& \left|\left|\left| \int_{0}^{s} U(s',s) \,\left(\H(\u(s')) - \H(\v(s'))\right) \,\d_{\theta} \v ds' \right|\right|\right| \\
	& \lesssim \omega^{-(m-1)}\, \beta^{-1}\e \,|||\u - \v|||\,|||\widetilde{\H}(\u,\v)|||\, |||\v||| \\
	& \lesssim \omega^{-(m-1)}\, \beta^{-1}\e \,|||\u - \v|||\,|||\widetilde{\H}(\u,\v)|||
\end{align*}

\noindent as $\v\in B_{\EE}(0,1)$, and recalling the prefactor $\e$ in notation \eqref{def.mathbf.H}. In particular, for $	H = A - \und{A} $ we have just for all $\u$ and $\v$ in $B_{\EE}(0,1)$ both
$$
	|||\H(\u)||| \lesssim |||\mathbf{A}(\u) - \und{\A}(\u)||| \quad \text{and} \quad |||\widetilde{\H}(\u,\v)||| \lesssim 1
$$

\noindent thanks to Lemma \ref{lemma.holomorphic}. Finally there holds for all $\u$ and $\v$ in $B_{\EE}(0,1)$:
\begin{eqnarray*}
	|||T^{[\theta]}(\u)-T^{[\theta]}(\v)||| & \lesssim & \omega^{-(m-1)}\, \beta^{-1} \left(|||\mathbf{A}(\u) - \und{\A}(\u)||| + \e \right) |||\u - \v||| .
\end{eqnarray*}

For both $T^{[x]}(\u) - T^{[x]}(\v)$ and $T^{[\u]}(\u) - T^{[\u]}(\v)$ we do the same to finally get
\begin{eqnarray*}
	||| T^{[x]}(\u) - T^{[x]}(\v) ||| & \lesssim & \omega^{-(m-1)}\,R\rho^{-1} |||\u - \v||| \\
	||| T^{[\u]}(\u) - T^{[\u]}(\v) |||& \lesssim & \omega^{-(m-1)}\, \beta^{-1}\e |||\u - \v|||
\end{eqnarray*}

\noindent as $\e$ is small.

\end{proof}

Thanks to Corollary \ref{reg_dtheta.bis}, we have a finer version of the contraction estimates:
\begin{coro}[Finer contraction estimates in $\EE$]
	\label{prop.estimates.bis}
	There are $R_0$, $\rho_0>0$ such that for all $ \beta >0$, $R\geq R_0$, $\rho>\rho_0$ and $\e \in (0,1)$, we get the following estimates for all $\u$ and $\v$ in $B_{\EE}(0,1)$:
	\begin{align}
		\label{esti.T.bis}
		& |||T(\u)||| \lesssim \omega^{-(m-1)}\,\left( \beta^{-1}\, \e|||\u||| + R\rho^{-1}\right) |||\u||| \\
		\label{esti.TT.bis}
		& |||T(\u) - T(\v)||| \lesssim \omega^{-(m-1)}\,\left( \beta^{-1} \, \e|||\u||| + R\rho^{-1}\right)  |||\u - \v||| .
	\end{align}
\end{coro}


\section{Existence of solutions and estimates from below}
\label{section.existence}



\subsection{Existence of solutions}


Thanks to the Corollary \ref{prop.estimates.bis}, we can now solve the fixed point equation \eqref{fixed.point.equation} in the ball $B_{\EE}\left(0,|||\f_{\e}|||\right)$, provided that $|||\f_{\e}||| \leq 1/2$:

\begin{coro}[Existence of solutions]
	\label{coro.fixedpoint} 
	Let $R(\e) > R_0$, $\rho(\e) > \rho_0$, $\beta(\e) >0$ and $\und{s}(\e)$ be such that 
	\be
		\label{hypo.coro}
		\lim_{\e\to 0} \omega^{-(m-1)}\,\left( \beta^{-1} \e|||\f_{\e}||| + R\rho^{-1}\right) = 0 .
	\ee
	\noindent Then for any $\e$ small enough, the fixed point equation \eqref{fixed.point.equation}, with $\f_{\e}$ defined by \eqref{free.solution}, has a unique solution $\u_{\e}$ in $B_{\EE(R,\rho)}\left(0,2|||\f_{\e}|||\right)$. This solution satisfies
	\be
		\label{esti.solution}
		|||\u_{\e} - \f_{\e}||| \lesssim \omega^{-(m-1)}\,\left( \beta^{-1} \e |||\f_{\e}||| + R\rho^{-1}\right)|||\f_{\e}||| 	.
	\ee
\end{coro}

The proof of the Corollary is straigthforward using the estimates of Corollary \ref{prop.estimates.bis}, under the condition of smallness \eqref{hypo.coro}. For convenience we introduce
\be
	\label{def.K_epsilon}
	K(\e) = \omega^{-(m-1)}\left( \beta^{-1} \, \e |||\f_{\e}||| + R\rho^{-1}\right).
\ee


\subsection{Bounds from below for the solutions}
\label{subsection.below}


Recall that in Section \ref{subsection.sketch}, we explained that to prove Hadamard instability, we prove first that the solution $\u_{\e}$ of \eqref{fixed.point.equation} has the same growth as $\f_{\e}$ given by Lemma \ref{lemma.growth.free.solution}. That is, the goal is to prove
\be
	\label{growth.solution}
	|\u_{\e}(s,x,\theta)| \gtrsim \omega^{-(m-1)}\,e^{-M} \exp \left( \int_{0}^{s} \g^{\flat}(\tau \,;r) d\tau \right)	\quad , \quad \forall\,(s,x,\theta) \in (\und{s} - 1 , \und{s})\times B_{r}(0) \times\mathbb{T} 
\ee

\noindent with $\g^{\flat}$ given by either \eqref{bound.g.flat} (under Assumption \ref{hypo.1}), \eqref{bound.g.flat.bis} (under Assumption \ref{hypo.2}) or \eqref{bound.g.flat.max} (under Assumptions \ref{hypo.2} and \ref{hypo.3}). It is indeed sufficient to prove this kind of estimate only on a small neighborhood of $(\und{s},0)\times\mathbb{T}$, and not on all the domain $\Omega_{R,\e\rho}(0)\times\mathbb{T} $. To this effect in view of Lemma \ref{lemma.growth.free.solution} it suffices to prove that
\be
	\label{esti.Ceps} 
	|(\u_{\e}-\f_{\e})(s,x,\theta)| \lesssim C(\e)\omega^{-(m-1)}\,e^{-M(\e)} \exp \left( \int_{0}^{s} \g^{\flat}(\tau \,;r) d\tau \right) 
\ee

\noindent for some constant $C(\e)$ such that $C(\e) \to 0$ as $\e \to 0$. The constant $C(\e)$ will depend on the parameters $M'$, $R$, $\rho$, $\beta$ and $\omega$. Finding suitable parameters such that $C(\e) \to 0$ as $\e \to 0$ will depend on under which Assumption we work, as it is precised in Propositions \ref{prop.below}, \ref{prop.below.bis} and \ref{prop.below.max}.

First, we decompose $\u_{\e}-\f_{\e}$ with its Fourier modes
$$
	(\u_{\e}-\f_{\e})(s,x,\theta) = \sum_{n\in\Z} (\u-\f_{\e})_n(s,x) e^{in\theta} .
$$

\noindent Thanks to the first property of Lemma \ref{majoring.properties} and estimate \eqref{esti.u}, for all $(s,x,\theta)\in\Omega_{R,\e\rho}(0)\times\mathbb{T}$ there holds 
\begin{eqnarray*}
	|(\u_{\e}-\f_{\e})(s,x,\theta)| & \leq & \sum_{n\in\Z} |(\u_{\e}-\f_{\e})_n|(s,x) \\
		& \leq & \dsp |||\u_{\e}-\f_{\e}||| \sum_{n\in\Z} \frac{\dsp c_1}{\dsp n^2+1} \exp\left(-\left(M'- \int_{0}^{s} \g(\tau) d\tau \right) \left<n\right> \right) \Phi\left(R|x|_1 + \e\rho s\right) 
\end{eqnarray*}

\noindent where $\g$ is defined in \eqref{def.g}. Then, as $M'- \int_{0}^{s} \g(\tau) d\tau >0$ for any $s\in[0,\und{s})$ (recall definition \eqref{def.und.s.1} of $\und{s}_1$ and definition \eqref{finaltime} of $\und{s}$) and $\left<n\right> \geq 1$ for all $n$, we have
\begin{eqnarray*}
	  |(\u_{\e}-\f_{\e})(s,x,\theta)| & \leq & \dsp |||\u_{\e} - \f_{\e}|||\,\exp\left(-\left(M'- \int_{0}^{s} \g(\tau) d\tau \right) \left<n\right> \right) \sum_{n\in\Z} \frac{\dsp c_1}{\dsp n^2+1} \,\Phi\left(R|x|_1 + \e\rho s\right) \\
	  & \leq & \dsp |||\u_{\e} - \f_{\e}|||\,\exp\left(-\left(M'- \int_{0}^{s} \g(\tau) d\tau \right) \left<n\right> \right) \sum_{n\in\Z} \frac{\dsp c_1}{\dsp n^2+1} \,\Phi(1) 
\end{eqnarray*}

\noindent and the last inequality holds because $\Phi$ is convergent in $1$. As the series of the right-hand side of the previous inequality is convergent, there holds
\begin{eqnarray*}	  
	  |(\u_{\e}-\f_{\e})(s,x,\theta)| & \lesssim & \dsp |||\u_{\e} - \f_{\e}|||\,\exp\left(-\left(M'- \int_{0}^{s} \g(\tau) d\tau \right) \left<n\right> \right) 
\end{eqnarray*}

\noindent for all $(s,x,\theta)\in\Omega_{R,\e\rho}(0)\times\mathbb{T}$. 

Next, by Lemma \ref{lemma.norm.free.solution}, estimate \eqref{hypo.coro} of Corollary \ref{coro.fixedpoint} and notation \eqref{def.K_epsilon}, we have successively
\begin{eqnarray}
	|(\u_{\e}-\f_{\e})(s,x,\theta)| & \lesssim & K(\e)\,|||\f_{\e}|||\,\exp\left(-\left(M'- \int_{0}^{s} \g(\tau) d\tau \right) \left<n\right> \right) \nonumber \\
		& \lesssim & K(\e)\, \omega^{-(m-1)} e^{M'-M(\e)}\,\exp\left(-\left(M'- \int_{0}^{s} \g(\tau) d\tau \right) \left<n\right> \right) \nonumber \\
	  & \lesssim & K(\e)\, \omega^{-(m-1)} e^{-M(\e)} \exp\left( \int_{0}^{s} \g(\tau) d\tau \right) \label{esti.local.2} 
\end{eqnarray}

\noindent using $\langle n \rangle \geq 1$ for all $n\in\Z$. Note that estimate \eqref{esti.local.2} holds pointwise for all $(s,x,\theta)\in\Omega_{R,\e\rho}(0)\times\mathbb{T} $. Now we focus our analysis to the smaller domain $(\und{s} - 1 , \und{s}) \times B_{r}(0) \times \mathbb{T}$. Having \eqref{esti.Ceps}  in mind, we rewrite \eqref{esti.local.2} to get
\begin{align}
	& |(\u_{\e}-\f_{\e})(s,x,\theta)| \nonumber \\
	& \lesssim K(\e)\,\exp\left( \int_{0}^{s} (\g(\tau) - \g^{\flat}(\tau\,;r) ) d\tau \right) \,\omega^{-(m-1)}\,e^{-M} \exp \left( \int_{0}^{s} \g^{\flat}(\tau\,;r) d\tau \right) \nonumber \\
	& \lesssim K(\e)\,\exp\left( \int_{0}^{\und{s}} (\g(\tau) - \g^{\flat}(\tau\,;r) ) d\tau \right) \,\omega^{-(m-1)}\,e^{-M} \exp \left( \int_{0}^{s} \g^{\flat}(\tau\,;r) d\tau \right) \\
	& \lesssim K(\e)\,\exp\left( \und{s}\,\beta + \int_{0}^{\und{s}} (\g^{\sharp}(\tau\,;R, \omega) - \g^{\flat}(\tau\,;r, \omega) ) d\tau \right) \,\omega^{-(m-1)}\,e^{-M} \exp \left( \int_{0}^{s} \g^{\flat}(\tau\,;r, \omega) d\tau \right) \label{esti.local} 
\end{align}

\noindent by definition \eqref{finaltime} of $\und{s}$ and definition \eqref{def.g} of $\g$. So to get \eqref{esti.Ceps} we need 
$$
	\lim_{\e\to0} \,K(\e)\,\exp\left( \und{s}\,\beta + \int_{0}^{\und{s}} (\g^{\sharp}(\tau\,;R, \omega) - \g^{\flat}(\tau\,;r, \omega) ) d\tau \right) = 0 .
$$

\noindent If $K(\e) \to 0$ as in \eqref{hypo.coro}, and as $\omega(\e)$ is a small parameter, it suffices then to have
\be
	\label{first.constraint}
	\lim_{\e\to0} \,\exp\left( \und{s}\,\beta + \int_{0}^{\und{s}} (\g^{\sharp}(\tau\,;R, \omega) - \g^{\flat}(\tau\,;r, \omega) ) d\tau \right) = 0 
\ee

\noindent which brings another constraint on the parameters, after \eqref{hypo.coro}. 

We recall also the constraint on the parameters $M'$ and $\rho$ coming from the competition between the growth time $\und{s}_1$ defined in \eqref{def.und.s.1} and the regularity time $(\e\rho)^{-1}$. To see the growth of the solution, we need it to exist on a sufficiently large time compared to the growth time, that is we need $\und{s} = \und{s}_1$. This is equivalent to
\be
	\label{second.constraint}
	\lim_{\e\to0} \,\und{s}_1\e\rho =0 .
\ee

A last constraint on the parameters comes from the smallness of the norm of the free solution, that is
\be
	\label{third.constraint}
	\lim_{\e \to0} \omega^{-(m-1)}\,e^{M' - M} = 0 
\ee

\noindent following Lemma \ref{lemma.norm.free.solution}.

In constraint \eqref{first.constraint}, recall that bound $\g^{\sharp}(\tau\,;R,\omega)$ is defined in Lemma \ref{lemma.growth.propa}. Under Assumption \ref{hypo.1}, the bound $\g^{\sharp}$ is given by \eqref{bound.g.sharp} ; under Assumption \ref{hypo.2}, by \eqref{bound.g.sharp.bis} ; and under Assumptions \ref{hypo.2} and \ref{hypo.3}, by \eqref{bound.g.sharp.max}. Similarly, recall that bound $\g^{\flat}(\tau\,;r,\omega)$ is defined in Lemma \ref{lemma.growth.free.solution}. Under Assumption \ref{hypo.1}, the bound $\g^{\flat}$ is given by \eqref{bound.g.flat} ; under Assumption \ref{hypo.2}, by \eqref{bound.g.flat.bis} ; and under Assumptions \ref{hypo.2} and \ref{hypo.3}, by \eqref{bound.g.flat.max}. In each case, we combine altogether constraints \eqref{hypo.coro}, \eqref{first.constraint}, \eqref{second.constraint} and \eqref{third.constraint}, and we give in the following three Propositions a choice of parameters satisfying those constraints.

\begin{prop}
	\label{prop.below}
	Under Assumption {\rm \ref{hypo.1}}, with the following choice of parameters 
	\be
		\label{good.parameters}
		\omega = \e^{\delta} , \quad \beta = \e^{\delta} , \quad R^{-1} = \e^{\delta} , \quad \rho^{-1} = \e^{(1 + (m-1)\delta)/2} , \quad M' = M(\e) - \min\{ 0 , 1 - (2m-1)\delta \}|\ln(\e)| 
	\ee
	
	\noindent and the limitation on the Gevrey index
	$$
		\sigma < \delta <1/(m+1)
	$$
	
	\noindent where $m$ is the algebraic multiplicity of $\lambda_0$, the fixed point equation \eqref{fixed.point.equation} has a unique solution $\u_{\e}$ in $\EE$ which satisfies 
	\be 
		|\u_{\e}(s,x,\theta)| \gtrsim \e^{-\delta(m-1)} e^{-M(\e)} \exp \left( \int_{0}^{s} \g^{\flat}(\tau\,; r, \omega) d\tau \right) \quad , \quad \forall\,(s,x,\theta)\in(\und{s} - 1, \und{s}) \times B_r(0)\times\mathbb{T} 
	\ee
	\noindent for any $r \lesssim \e^{\delta}$. Another consequence of \eqref{good.parameters} is 
	\be
		\label{final.und.s}
		\und{s} \approx \e^{-\delta} .
	\ee

\end{prop}

\begin{proof}
	It is straightforward to verify that parameters given by \eqref{good.parameters} satisfy the four constraints \eqref{hypo.coro}, \eqref{first.constraint}, \eqref{second.constraint} and \eqref{third.constraint}. The aim of the proof is to show that those parameters are optimal, in some sens. For that, we assume that the constraints are satisfy and we get constraints directly on $M'$, $\rho$, $R$, $\omega$ and $\beta$. 
	
	First, \eqref{second.constraint} being satisfied the final time is
	$$
		\und{s} = \und{s}_1
	$$
	
	\noindent defined by \eqref{def.und.s.1}. In the asymptotic $\e \to 0$ there holds
	\begin{eqnarray*}
		\int_{0}^{\und{s}_1} \g(\tau) d\tau & \sim & \und{s}_1 \g(\und{s_1}) \\
			& \approx & \g_0 \und{s}_1
	\end{eqnarray*}
	
	\noindent which implies that 
	$$
		\und{s}_1 \approx \frac{M'}{\g_0} .
	$$
	
	\noindent Constraint \eqref{third.constraint} implies that $M' - M = {}- c(\e) + (m-1)\ln\omega$ with $\lim_{\e\to 0} c(\e) = +\infty$. We assume that $c(\e) = o\left(\e^{-\delta}\right)$ to get $M' \sim M$, hence
	$$
		\und{s} \approx M = \e^{-\delta}.
	$$
	
	\noindent We also rewrite \eqref{second.constraint} as
	\be
		\label{local.1}
		\lim_{\e\to0}\e^{1-\delta} \rho = 0.
	\ee
		
	Second, we focus on \eqref{first.constraint}. By definitions \eqref{bound.g.sharp} and \eqref{bound.g.flat} we have
	$$
		\int_{0}^{\und{s}} ( \g^{\sharp}(\tau\, ; R,\omega) - \g^{\flat}(\tau\, ; r,\omega) ) d\tau \lesssim \und{s}\left(\e \und{s} + R^{-1} + r + \omega \right) .
	$$
	
	\noindent As $\und{s} \approx \e^{-\delta}$, for \eqref{first.constraint} to be satisfied we need $ \und{s}\left(\beta + \e \und{s} + R^{-1} + r + \omega \right) $ to be bounded, hence the choices
	$$
		\beta = \e^{\delta} \quad , \quad r = \e^{\delta} \quad , \quad \omega = \e^{\delta}
	$$
	
	\noindent and the constraints
	\be
		\label{local.cons}
		\e \und{s}^2 \lesssim 1 \quad , \quad R^{-1} \lesssim \e^{\delta} .
	\ee
	
	\noindent The first one implies in particular
	$$
		\delta < 1/2 .
	$$
	
	The constraint \eqref{hypo.coro} is now
	$$
		\lim_{\e\to0} \e^{-\delta(m-1)} \left( \e^{1-\delta} \e^{-\delta(m-1)} e^{M'-M} + R\rho^{-1} \right) = 0
	$$
	
	\noindent using \eqref{norm.free.solution}, and that is equivalent to both
	$$
		\lim_{\e\to0} e^{M'-M} \e^{1- \delta(2m-1)} = 0 \quad \text{and} \quad \lim_{\e\to0} \e^{-\delta(m-1)} R\rho^{-1} = 0.
	$$
	
	\noindent The first limit leads to the choice
	$$
		M' = M - \min\{ 0 , 1 - (2m-1)\delta \}|\ln(\e)|
	$$
	
	\noindent reminding that $\delta\in(0,1/m)$. The second limit, combined with \eqref{local.1}, gives us 
	\be
		\label{local.voila.2}
		\e^{1-\delta} \ll \rho^{-1} \ll \e^{\delta(m-1)} R^{-1}
	\ee
	
	\noindent using notation \eqref{notation.ll}. We note then that in particular, $R^{-1}$ has to be greater than $\e^{1-m\delta}$. As $R^{-1}$ has to be also smaller than $\e^{\delta}$, it implies the limitation 
	\be
		\label{local.voila}
		\e^{1-\delta} \ll \e^{\delta(m-1)} \e^{\delta} 
	\ee
	
	\noindent which is equivalent to
	$$
		\delta <1/(m+1) ,
	$$
	
	\noindent compatible with the previous limitation $\delta <1/2$ as $m \geq 1$.
	
\end{proof}

\begin{prop}
	\label{prop.below.bis}
	Under Assumption {\rm \ref{hypo.2}}, with the following choice of parameters 
	\be
		\label{good.parameters.bis}
		\omega = 0, \quad \beta = \e^{\delta} , \quad R^{-1} = \e , \quad \rho^{-1} = \e^{1-\delta/2} , \quad M' = M(\e) - (1-\delta)|\ln(\e)| 
	\ee
	
	\noindent and the limitation on the Gevrey index
	$$
		\sigma < \delta <1/2
	$$
	
	\noindent the fixed point equation \eqref{fixed.point.equation} has a unique solution $\u_{\e}$ in $\EE$ which satisfies 
	\be 
		|\u_{\e}(s,x,\theta)| \gtrsim e^{-M(\e)} \exp \left( \int_{0}^{s} \g^{\flat}(\tau\,; r, \omega) d\tau \right) \quad , \quad \forall\,(s,x,\theta)\in(\und{s} - 1, \und{s}) \times B_r(0)\times\mathbb{T} 
	\ee
	\noindent for any $r \lesssim \e^{\delta}$. Another consequence of \eqref{good.parameters} is 
	\be
		\label{final.und.s.bis}
		\und{s} \approx \e^{-\delta} .
	\ee

\end{prop}

\begin{proof}

	The proof is the same the one of Proposition \ref{prop.below}, with the difference that with Assumption \ref{hypo.2}, estimate \eqref{local.voila.2} is replaced by $		\e^{1-\delta} \ll \rho^{-1} \ll R^{-1}$ as $m=1$. Hence constraint \eqref{local.voila} is now $ \e^{1-\delta} \ll \e^{\delta(m-1)} \e^{\delta}  $ which is equivalent to $\delta <1/2$.

	%
	%
\end{proof}

\begin{prop}
	\label{prop.below.max}
	Under Assumptions {\rm \ref{hypo.2}} and {\rm \ref{hypo.3}}, with the following choice of parameters 
	\be
		\label{good.parameters.max}
		\omega = 1 , \quad \beta = \e^{\delta} , \quad R^{-1} = \e , \quad \rho^{-1} = \e^{1-\delta/2} , \quad M' = M(\e) - (1-\delta)|\ln(\e)| 
	\ee
	
	\noindent and the limitation on the Gevrey index
	$$
		\sigma < \delta <2/3
	$$
	
	\noindent the fixed point equation \eqref{fixed.point.equation} has a unique solution $\u_{\e}$ in $\EE$ which satisfies 
	\be 
		|\u_{\e}(s,x,\theta)| \gtrsim e^{-M(\e)} \exp \left( \int_{0}^{s} \g^{\flat}(\tau\,; r, \omega) d\tau \right) \quad , \quad \forall\,(s,x,\theta)\in(\und{s} - 1, \und{s}) \times B_r(0)\times\mathbb{T} 
	\ee
	\noindent for any $r \lesssim \e^{\delta}$. Another consequence of \eqref{good.parameters} is 
	\be
		\label{final.und.s.max}
		\und{s} \approx \e^{-\delta} .
	\ee

\end{prop}

\begin{proof}
	The proof is the same the one of Proposition \ref{prop.below}, with the difference that with Assumption \ref{hypo.2}, the bounds \eqref{bound.g.sharp} and \eqref{bound.g.flat} are replaced by the sharper bounds \eqref{bound.g.sharp.max} and \eqref{bound.g.flat.max}, respectively. First, note that the parameter of trigonalization $\omega$ does not appear anymore, and is then taken equal to one. Second, thanks to Assumption \ref{hypo.2}, difference $\g^{\sharp} - \g^{\flat}$ is improved:
	\be
		\label{proof.local.cool.2}
		\g^{\sharp}(\tau\, ; R,\omega) - \g^{\flat}(\tau\, ; r,\omega) \lesssim \e^{2}\und{s}^{2} + r
	\ee
	
	\noindent This implies in particular that 
	$$
		\int_{0}^{\und{s}} ( \g^{\sharp}(\tau\, ; R,\omega) - \g^{\flat}(\tau\, ; r,\omega) ) d\tau \lesssim \und{s}\left( r + \e^{2}\und{s}^2 \right) 
	$$
	
	\noindent which no longer implies constraints \eqref{local.cons}. It suffices then to follow the rest of the proof of Proposition \ref{prop.below}.
\end{proof}

\begin{remark}
	\label{remark.amelioration}
	
	Estimate \eqref{proof.local.cool.2} in the previous proof shows that the limiting Gevrey index increases as $\g^{\sharp} - \g^{\flat}$ decreases (with $\g^{\sharp}$ and $\g^{\flat}$ the upper and lower rates of growth introduced in Lemmas {\rm \ref{lemma.growth.propa}} and {\rm \ref{lemma.growth.free.solution}}). In particular, if the distinguished eigenvalue $\lambda$ is very flat at the distinguished point $(0,x_0)$, then the limiting Gevrey index is close to $1$, as claimed in Remark {\rm \ref{remark.aprestheo3}}.
	
\end{remark}

\section{Conclusion: Hadamard instability in Gevrey spaces}
\label{section.Hadamard}


To close the proofs of Theorems \ref{theorem.2}, \ref{theorem.3} and \ref{theorem.4} we have now to get an estimate of the ratio
$$ \frac{\dsp ||u_{\e}||_{L^2(\Omega_{R,\rho}(0))}}{\dsp ||h_{\e}||^{\a}_{\sigma,c,K}} $$

\noindent The previous Sections show the existence of a family of solutions $\u$ starting from $\f_{\e}$ of the fixed point equation \eqref{fixed.point.equation}. Thanks to the ansatz \eqref{ansatz} which we recall here
$$
	u_{\e}(t,x) = \e\u(\e^{-1}\,t,x,x\cdot\xi_0/\e)
$$

\noindent we have then a family of solutions $u_{\e}$ existing in domains $\Omega_{R,\rho}(0)$, with $R$ and $\rho$ given by \eqref{good.parameters}. As $\und{s} < (\e\rho)^{-1}$ the domain of regularity $\Omega_{R,\rho}(0)$ for $\u$ contains the cube of size $\e$
$$ C_{\e} = \{ (t,x) \,|\, \e\und{s} - \e < t <\e\und{s} ,\quad |x| < \e \} $$

\noindent On one hand, thanks to estimate \eqref{growth.solution} with $r=\e$ there holds 
\begin{eqnarray*}
	||u_{\e}||_{L^2(\Omega_{R,\rho})} & \geq & ||u_{\e}||_{L^2(C_{\e})} \\
	 & \gtrsim & \inf_{\e\und{s} - \e < t <\e\und{s}} \left(\e^{-\delta(m-1)}  e^{-M(\e)} \exp \left( \int_{0}^{t/\e} \g^{\flat}( \tau/\e) d\tau \right) \right) \,||1||_{L^2(C_{\e})} \\
	 & \gtrsim & \e^{-\delta(m-1)} e^{-M(\e)} \exp \left( (\und{s}-1) \left( \g_0 - \e\und{s} - r - \omega \right)  \right) \, \e^{(d+1)/2} \\
	 & \gtrsim & \e^{-\delta(m-1)} e^{-M(\e)} e^{\g_0\und{s}}\, \e^{(d+1)/2}
\end{eqnarray*}

\noindent Next, by choice of $M' = M -(m\delta-1)|\ln(\e)|$ we get
$$
	||u_{\e}||_{L^2(\Omega_{R,\rho})} \gtrsim \e^{-\delta(2m+1)+1} e^{-M'(\e)} e^{\g_0\und{s}} \, \e^{1+(d+1)/2}.
$$

\noindent As 
$$
	M' = \und{s} \g = \und{s}\g_0 (1+ 2\e^{\delta})
$$

\noindent this implies that 
\begin{eqnarray*}
	||u_{\e}||_{L^2(\Omega_{R,\rho})} & \gtrsim & e^{ - \und{s}\g_0(1 + 2\e^{\delta}) + \g_0 \und{s}} \, \e^{1+(d+1)/2-\delta(2m+1)} \\
		& \gtrsim & \e^{1+(d+1)/2-\delta(2m+1)}
\end{eqnarray*}

\noindent as $\und{s}\e^{\delta} \approx 1$.

On the other hand, by Lemma \ref{size.gevrey.exp} and definition \eqref{size.gevrey} of $M$ there holds
$$ 
	||h_{\e}||_{\sigma,c,K} \lesssim \e e^{-M} e^{c\e^{-\sigma}} = \e \exp(c\e^{-\sigma} - \e^{-\delta}) 
$$

\noindent which is small as soon as $\sigma < \delta$. Combining those two estimates we have then
$$ 
	\frac{\dsp ||u_{\e}||_{L^2(\Omega_{R,\rho})}}{\dsp ||h_{\e}||^{\a}_{\sigma,c,K}} \gtrsim \e^{1+(d+1)/2-\delta(2m+1)-\a} \exp(-\a c\e^{-\sigma} + \a \e^{-\delta})  
$$
\noindent that tends to $+\infty$ as $\e\to0$ because $\sigma < \delta$ no matter whether $1+(d+1)/2-\delta(2m+1)-\a$ is positive or negative , which ends the proof of Theorem \ref{theorem.2}. 

The proofs of Theorems \ref{theorem.3} and \ref{theorem.4} rely on the exact same computations, using Proposition \ref{prop.below.bis} and Proposition \ref{prop.below.max} respectively, instead of Proposition \ref{prop.below}.

	\bibliographystyle{alpha}
	\bibliography{onset_insta}


\end{document}